\documentclass{article}

\usepackage[inner=1in,outer=1in,top=1in,bottom=1in,
marginparwidth=.7in,marginparsep=.1in]{geometry}

\usepackage[labelfont=bf]{caption}
\usepackage{graphicx}
\usepackage{amssymb}
\usepackage{amsmath}
\usepackage{amsthm}
\usepackage{textcomp}
\usepackage{gensymb}
\usepackage{hyperref}
\usepackage{cleveref}
\usepackage{authblk}
\usepackage{mathtools}
\usepackage{color}
\usepackage{url}
\usepackage{xcolor}

\newcommand{\abs}[1]{\lvert#1\rvert}

\allowdisplaybreaks

\usepackage[sorting=none, style=nature, url=false, doi=true, isbn=false]{biblatex}
\graphicspath{{./}}

\newtheorem{theorem}{Theorem}[section]
\newtheorem{observation}{Observation}
\newtheorem{proposition}{Proposition}

\newcommand{\volterra}{\mathcal{V}}
\DeclareMathOperator*{\argmin}{arg\,min}

\begin{document}

\title{Fast rigid rotational alignment of voxelized 3-D density maps in sliced Wasserstein distance}

\author[1]{Oscar Mickelin\thanks{oscarmi@tsinghua.edu.cn}}
\author[2]{Yunpeng Shi\thanks{ypshi@ucdavis.edu}}
\author[3,4]{Amit Singer\thanks{amits@math.princeton.edu}}
\author[5,6]{Eric J. Verbeke\thanks{ev9102@princeton.edu}}

\affil[1]{
Yau Mathematical Sciences Center, Tsinghua University,
Beijing 100084, China
}

\affil[2]{Department of Mathematics, University of California at Davis, Davis CA, USA}

\affil[3]{Department of Mathematics, Princeton University, Princeton, NJ, USA}

\affil[4]{Program in Applied and Computational Mathematics, Princeton University, Princeton, NJ, USA}

\affil[5]{Center for Statistics and Machine Learning, Princeton University, Princeton, NJ, USA}

\affil[6]{Premium Research Institute for Human Metaverse Medicine, The University of Osaka, Japan}

\date{}

\maketitle

\vspace{-0.5cm}

\begin{abstract}
\noindent We develop a fast method for the rigid alignment of voxelized three-dimensional density functions over $\mathsf{SO}(3)$ using sliced optimal transport. Our approach efficiently transforms each density map into a tensor of one-dimensional quantile functions, taken from central slices over the unit sphere, such that the Euclidean distance of the transformed volumes approximates the sliced 2-Wasserstein distance. Our formulation is therefore compatible with existing, fast, voxel-based Euclidean alignment algorithms. We provide theoretical analysis and numerical experiments demonstrating that the sliced 2-Wasserstein distance is a robust metric for comparing three-dimensional shapes and for computing their rigid alignment.
\end{abstract}




\section{Introduction}
\label{sec:intro}

We consider the problem of three-dimensional (3-D) rigid alignment between density functions $f$ and $g$ which are defined on $N \times N \times N$ Cartesian voxel grids, and are supported on the closed unit ball $\mathbb{B} = \{x \in \mathbb{R}^3 : \|x\|_2 \leq 1\}$. Additionally, the functions are assumed to be non-negative with $f \geq 0, g \geq 0$ and normalized to $\int_\mathbb{B} f(x)dx = \int_\mathbb{B} g(x) = 1$, such that they may be thought of as probability density functions and are therefore compatible with optimal transport frameworks described later. Specifically, we seek to find:

\begin{equation}
\label{eq:rotational_distance}
\argmin_{R \in \mathsf{SO}(3)} d(f , Rg) ,
\end{equation}
where $R \in \mathsf{SO}(3)$ is a rotation and $d(\cdot,\cdot) \in [0,\infty)$ denotes a generic distance function. In particular, we are motivated by the case where $f \neq Rg$ for any $R \in \mathsf{SO}(3)$, which we refer to as~\textit{heterogeneous} functions (or heterogeneous 3-D volumes). This type of problem arises in many applications, most prominently in 3-D modeling~\cite{ZhirongWu3D2015} and biological imaging methods such as electron cryo-microscopy (cryo-EM)~\cite{SingerComputational2020}. Typically,~\cref{eq:rotational_distance} is solved using the Euclidean distance or standard correlation as the metric for comparison. However, it is known that the landscape with respect to these metrics can be highly irregular, especially for heterogeneous volumes~\cite{SingerAlignment2024}.

We therefore focus on using metrics from optimal transport, as they are known to be more robust to changes in the geometry of the objects being compared~\cite{VillaniOptimal2009, PeyreComputational2019}. However, these metrics, in particular the Wasserstein distance (or Kantorovich-Rubenstein distance), come with a high computational cost. The naive computational complexity for the 2-Wasserstein distance between two $N \times N \times N$ functions using linear programming is $\mathcal{O}(N^9)$ operations. For rotational alignment, where the distance must be computed over a grid of rotations (typically $N^3$ rotations, to scale with the size of the input functions), the computational complexity increases to $\mathcal{O}(N^{12})$ operations. Although there exist fast algorithms for approximating the Wasserstein distance~\cite{CuturiSinkhorn2013}, especially on geometric domains~\cite{SolomonConvolutional2015}, they are still too slow for large-scale and high-dimensional 3-D alignments. This high computational cost drastically limits the application of the Wasserstein distance to rigid alignment tasks. Meanwhile, there exist efficient algorithms for computing the correlation over a grid of rotations in $\mathsf{SO}(3)$ with complexity $\mathcal{O}(N^4)$ operations~\cite{KovacsFast2002, KyatkinAlgorithms2000}. Therefore, this work attempts to unify the two frameworks of fast rotational correlations in $\mathsf{SO}(3)$ with optimal transport metrics. Our contributions are twofold:

\begin{enumerate}
    \item We detail an optimal transport-based alignment algorithm that performs a full grid optimization over $\mathsf{SO}(3)$ in $\mathcal{O}(N^4)$ operations. Interestingly, we show how selecting $d$ in~\cref{eq:rotational_distance} as the sliced 2-Wasserstein distance leads to a problem of aligning two helper quantities in the Euclidean distance. Many  alignment algorithms exist for the Euclidean distance and our approach therefore results in a modular and efficient framework where improved algorithms for Euclidean alignment can be leveraged to the setting of alignment in sliced 2-Wasserstein metrics.

    As an additional advantage, our alignment algorithm works directly on voxel-representations of the data, unlike several approaches that require an initial conversion to point-cloud representations.
    Lastly, the algorithm is deterministic and does not require sufficiently accurate initializations or sufficiently tuned hyperparameters, whereas many related algorithms (discussed later) do.    
    \item We provide two theorems about the stability of three-dimensional sliced Wasserstein metrics between a structure and a deformation of itself. These results can be used to inform alignment of two heterogeneous functions, since the distance of the aligned structures results in a lower bound on the degree of the heterogeneity through the stability theorems.
\end{enumerate}

This work builds on~\cite{ShiFast2026} which studied sliced Wasserstein-based alignment algorithms for two-dimensional functions. Compared to the 2-D case, the 3-D case poses additional challenges since the higher dimension adds a degree of freedom to the parametrization of the rotations. This means that we cannot simply apply a variation of the 2-D algorithm and care must be taken both theoretically and implementation-wise to obtain fast and accurate algorithms. We therefore phrase our results in terms of a more general and modular framework, as discussed in bullet point $1$ above. 
We show numerical results that highlight the advantage of our algorithm over traditional correlation-based Euclidean alignment methods and a selection of other optimal transport-based alignment algorithms. Our code is available at \url{[https://github.com/EricVerbeke/fast_volume_alignment_ot]}.

\section{Related works}
\label{sec:related_works}




Three-dimensional rigid registration is a classical problem in computer vision.
Existing methods can broadly be categorized according to the representation of the data and the optimization strategy. For example, two common approaches are point cloud registration algorithms~\cite{BeslMethod1992, BingJianRobust2011, KhooNonIterative2016}, and algorithms that use 3-D feature descriptors for alignment~\cite{HutchisonUnique2010}. Our approach focuses on using voxelized data directly. Voxel-based alignment methods are typically computed by evaluating rotational correlations using a spherical harmonic basis~\cite{KovacsFast2002, KyatkinAlgorithms2000}. While this approach is successful and widely used, it does not specifically handle deformations in the data. We note that there are methods that address finding correspondences between heterogeneous shapes~\cite{OvsjanikovFunctional2012}, but they typically assume some kind of global registration of the input data. Therefore, the main contribution of our work is to combine the fast voxel-based alignment methods with optimal transport metrics to provide a more robust rigid alignment framework for heterogeneous data.

Specifically, we utilize a sliced Wasserstein approach in this work. The sliced Wasserstein distance was originally introduced for 2-D applications in~\cite{RabinWasserstein2012} and has since been extended more generally in~\cite{KolouriGeneralized2019, KolouriGeneralized2022}. It has been used for many applications, particularly for 2-D and 3-D classification tasks~\cite{KolouriRadon2016, BeckmannGeneralizations2026}. We also note that there are already variants available for computing the sliced Wasserstein distance on the sphere. However, there are key differences between our work and existing methods. In our work, we define slices as regularly spaced lines passing though the origin with data projected orthogonally onto these lines. In contrast, other related works take slices to be 2-D planes passing through the unit sphere~\cite{QuellmalzSliced2023}, or consider the data to be projected onto great circles that are defined on the sphere~\cite{BonetSpherical2022} or into a Euclidean space using stereographic projections~\cite{TranStereographic2024}. Another related approach uses lines through the origin as slices, together with a Monte Carlo approach to compute the sliced Wasserstein distance. However, for our purpose of alignment, there is no advantage to random sampling compared to equispaced projections~\cite{ShiFast2026}.

A major application of 3-D alignment algorithms comes from the field of structural biology. In particular, cryo-EM relies heavily on alignment algorithms. Briefly, cryo-EM is a method used to obtain high-resolution 3-D reconstructions of biological molecules from large, high-dimensional and noisy 2-D tomographic projection images~\cite{BendorySingleParticle2020, SingerComputational2020}. Furthermore, cryo-EM experiments can contain extremely heterogeneous data; it can be both compositionally heterogeneous~\cite{KastritisCapturing2017, VerbekeClassification2018} (the data contains different underlying 3-D structures), and conformationally heterogeneous~\cite{ZhongCryoDRGN2021, GillesCryoEM2025} (the data contains the same 3-D structure but with possible deformations). Both forms of heterogeneity challenge conventional alignment algorithms. To address these specific challenges, many novel algorithms have been developed~\cite{PettersenUCSF2021, HarpazThreedimensional2023, HeAccurate2024}. Recently, several optimal transport based frameworks have also been introduced~\cite{SingerAlignment2024, RiahiAlignOT2023, TajmirRiahiAlignment2025}. Importantly, several key differences exist between our optimal transport based alignment algorithm and others; 1) we operate on the voxels directly, 2) we compute a full grid of rotations in optimal complexity and 3) we require no hyperparameters.

\section{Preliminaries}
This section details the definitions, conventions and observations used in the remainder of the paper, to make the exposition self-contained.
\subsection{Rotation conventions}\label{sec:rot_conv}
For any rotation $R \in \mathsf{SO}(3)$ and function $f: \mathbb{R}^3 \rightarrow \mathbb{R}$, we denote the action of $R$ on $f$ by $Rf:  \mathbb{R}^3 \rightarrow \mathbb{R}$, defined by
$$
(Rf)(x,y,z) = f(R^\top(x,y,z)).
$$

Let $\theta \in [0,\pi]$, $\varphi \in [0,2\pi]$, $\omega \in [0,2\pi]$ be three angles. We will use the ZYZ Euler angle representation of rotations, and therefore write any rotation in $\mathsf{SO}(3)$ as
$$
R_{\omega, \theta, \varphi} = R_\omega R_{\theta,\varphi} ,
$$
where $R_\omega$ is a rotation around the $z$-axis by the angle $\omega \in (0,2\pi]$, and $R_{\theta,\varphi}$ a rigid rotation mapping the vector on the unit sphere with angular coordinates $(\theta, \varphi) \in [0,\theta]\times (0,2\pi]$ to the positive $z$-axis. We can write $R_\omega = Z(\omega)$ and $R_{\theta, \varphi} := Y(-\theta)Z(-\varphi)$, where the two elemental rotations $Y$ and $Z$ are defined by
$$
Y(\theta) := \begin{bmatrix}
    \cos \theta & 0 & \sin \theta  \\
    0 & 1 & 0 \\
    -\sin \theta & 0 & \cos \theta
\end{bmatrix}, \quad Z(\varphi) := \begin{bmatrix}
    \cos \varphi & -\sin \varphi & 0 \\
    \sin \varphi & \cos \varphi & 0 \\
    0 & 0 & 1
\end{bmatrix} .
$$

We will in the following deal with rotations of this form, followed by projections onto the $z$-axis, which are invariant to the action of the $R_\omega$ (see Observation~\ref{obs:inv_omega} below for a precise statement).

We will, by a slight abuse of notation, write expressions like $R^\top(\theta, \varphi)$ to be the point on the unit sphere obtained by rotating the point with spherical angles $\theta, \varphi$. 

\subsection{The Radon Transform}
This section defines the Radon transform and recalls a number of its properties which will be useful in the remainder of the article. In three-dimensional space, the Radon transform is parametrized by the rotation angles $(\theta, \varphi)$. It is computed by projecting the plane orthogonal to the viewing direction onto the viewing direction axis. In more detail, for a given function $f: \mathbb{R}^3 \rightarrow \mathbb{C}$, we denote the result of applying the Radon transform to $f$ by $P_{\theta, \varphi}f : \mathbb{R} \rightarrow \mathbb{C}$, defined by
\begin{equation}
\label{eq:proj_operator}
\left(P_{\theta,\varphi} f\right)(z) \coloneq  \iint_{\mathbb{R}^2} (R_{\theta,\varphi} f)(x,y,z) dx dy  = \iint_{\mathbb{R}^2} f(R_{\theta,\varphi}^\top (x,y,z) ) dx dy .
\end{equation}

By allowing the spherical angles $\theta, \varphi$ to vary, we will (when convenient) alternatively view $P_{\theta, \varphi}f(z)$ as a function defined on $\mathbb{R}^3$ written in terms of the spherical coordinates $(\theta, \varphi, z) \in [0,\theta]\times (0, 2\pi] \times \mathbb{R}$. For convenience, we allow $z<0$ in what follows and note that the point with spherical coordinates $(\theta, \varphi, -\abs{z})$ can also be written with the spherical coordinates $(\pi-\theta, \pi+\varphi,\abs{z})$.

\begin{observation}\label{obs:nufft}
By the Fourier slice theorem, the projections can be evaluated using Fourier transforms
\begin{equation}
\label{eq:fourier_slice}
P_{\theta,\varphi}f(z) = \mathcal{F}_1^{-1}\left ((S_{\theta, \varphi}\mathcal{F}_3{f})(\xi) \right)= \int_\mathbb{R} (\mathcal{F}_3f)(R_{\theta,\varphi}^\top(0,0,\xi)^\top) e^{2\pi i \xi z}d\xi,
\end{equation}
where $\mathcal{F}_d$ denotes the $d$-dimensional Fourier transform with the convention
$$
\mathcal{F}_d(f)(\xi) := \int_{\mathbb{R}^d} f(x) e^{-2\pi i \xi \cdot x} dx, \quad  \xi\in \mathbb{R}^d, \qquad \mathcal{F}^{-1}_d(f)(x) := \int_{\mathbb{R}^d} f(\xi) e^{2\pi i \xi \cdot x} d\xi, \quad  x\in \mathbb{R}^d,
$$
and for any $g:\mathbb{R}^3 \rightarrow \mathbb{C}$, we define the function $\mathcal{S}_{\theta,\varphi}(g): \mathbb{R} \rightarrow \mathbb{C} $ by
$$
(\mathcal{S}_{\theta,\varphi} g)(k) = g(k\sin\theta\cos\varphi, k\sin\theta\sin\varphi,k\cos\theta)
$$
as the one-dimensional function obtained by restriction to given spherical angles $(\theta, \varphi)$.
\end{observation}

The next observation explains why our definition of $P_{\theta,\varphi}$ does not include the additional angle $\omega$. 
\begin{observation}\label{obs:inv_omega}
    For any $\omega \in (0,2\pi]$ and $f:\mathbb{R}^3 \rightarrow \mathbb{R}$
$$
\iint_{\mathbb{R}^2} (R_\omega R_{\theta,\varphi} f)(x,y,z) dx dy  = \iint_{\mathbb{R}^2} (R_{\theta,\varphi} f)(x,y,z)  dx dy .
$$
\end{observation}
\begin{proof}
    We have
$$
\iint_{\mathbb{R}^2} (R_\omega R_{\theta,\varphi} f)(x,y,z) dx dy = \iint_{\mathbb{R}^2} (R_{\theta, \varphi} f)(R_\omega^\top (x,y,z) ) dx dy = \iint_{\mathbb{R}^2} (R_{\theta, \varphi} f)(x',y',z) dx' dy',
$$
after performing the change of variables $(x',y',z) = R_{\omega}^\top(x,y,z)$ which has a Jacobian with determinant equal to $1$.
\end{proof}

Any rotation $R \in SO(3)$ defines an action on $P_{\theta, \varphi}f$ by rotating the spherical angles $(\theta , \varphi)$, i.e., by
$$
(RP_{\theta, \varphi}f)(z) = (P_{R^T(\theta, \varphi)}f)(z).
$$
Explicitly, we can describe this action with the following observation.

\begin{observation}\label{obs:3}
For any rotation $R \in SO(3)$,
    $(RP_{\theta, \varphi}f)(z) = P_{\theta, \varphi}(Rf)(z)$.
\end{observation}
\begin{proof}
For any $(\gamma, \alpha, \beta) \in (0,2\pi]\times [0,\pi]\times (0,2\pi]$, we have $R_{\gamma, \alpha, \beta}P_{\theta, \varphi}f(z) = P_{s,t}f$, where $(s,t,z)$ are the spherical coordinates of the image under $R_{\gamma, \alpha, \beta}^\top$ of the point with spherical coordinates $(\theta , \varphi,z)$, i.e., the spherical coordinates of $R_{\gamma, \alpha, \beta}^\top R_{\theta,\varphi}^\top e_z$. On the other hand, we have
\begin{align}
\left(P_{\theta,\varphi} (R_{\gamma, \alpha, \beta}f)\right)(z) & =  \iint_{\mathbb{R}^2} (R_{\theta,\varphi}R_{\gamma, \alpha, \beta} f)(x,y,z) dx dy 
\end{align}

By decomposing the rotation $R_{\theta,\varphi}R_{\gamma, \alpha, \beta}$ into Euler angles in the $ZYZ$ convention, we can write
$$
R_{\theta,\varphi}R_{\gamma, \alpha, \beta} = R_{u,v,w},
$$
for some Euler angles $u,v,w$. By Observation~\ref{obs:inv_omega}, this implies that
\begin{align}\label{eq:yzy2zyz}
\left(P_{\theta,\varphi} (R_{\gamma,\alpha, \beta}f)\right)(z) = \left(P_{v,w} f\right)(z),
\end{align}
and we now claim that $(v,w) = (s,t)$. To see this, note that the definition of $R_{v,w}$ in Section~\ref{sec:rot_conv} implies that $R_{v,w} (v,w) = e_z,$ so $$(v,w) = R_{v,w}^\top e_z = R_{v,w}^\top R_u^\top e_z = R_{u,v,w}^\top e_z = (R_{\theta,\varphi}R_{\gamma, \alpha, \beta})^\top e_z = R_{\gamma, \alpha, \beta}^\top R_{\theta,\varphi}^\top e_z = (s,t).$$


\end{proof}

\section{Algorithm}
This section presents the main result of the paper, i.e., an algorithm for aligning volumes using the sliced Wasserstein distance. For clarity, we separate the exposition into the following two subsections. The first section details the computations involved in computing and aligning in the sliced Wasserstein distance for continuous functions. The second section introduces the approximations involved when performing these computations on discrete data.
\subsection{Continuous algorithm}
The algorithm for alignment of continuous functions can be divided into three steps, presented below.
\label{sec:cont} 



\subsubsection{Step 1: computing Radon transforms}
Let $f,g : \mathbb{R}^3 \to \mathbb{R}^+$ be two probability measures which are normalized such that $\int_{\mathbb{R}^3} f(x,y,z) dx dy dz = \int_{\mathbb{R}^3} g(x, y, z) dx dy dz = 1$. We compute the Radon transforms of $f$ and $g$, parametrized by a rotation axis $(\theta, \varphi)$. For each axis $(\theta,\varphi) \in \times [0,\pi]\times [0,2\pi]$, we denote these one-dimensional projections by $\mu_{\theta,\varphi},\nu_{\theta,\varphi}:  \mathbb{R}  \rightarrow \mathbb{R}^+$. They are defined by
$$
\mu_{\theta,\varphi}(x) \coloneq \left(P_{\theta,\varphi} f\right)(x) , \qquad \text{ and } \qquad \nu_{\theta,\varphi}(x) \coloneq \left(P_{\theta,\varphi} g\right)(x). 
$$

\subsubsection{Step 2: computing Wasserstein distances}
The continuous sliced $2$-Wasserstein distance between $f$ and $g$ is the mean squared 1-D 2-Wasserstein distance over all projection angles
\begin{equation}
\label{eq:sliced_wasserstein}
d_{SW_2}^2 (f,g) = \frac{1}{4\pi} \int_0^{\pi}\int_0^{2\pi} d_{W_2}^2(\mu_{\theta,\varphi}, \nu_{\theta,\varphi}) \sin \theta d\varphi d\theta .
\end{equation}
Here, the squared 1-D 2-Wasserstein distance between two projections $\mu$ and $\nu$ is defined by
\begin{equation}
\label{eq:1d_wasserstein}
d_{W_2}^2 (\mu , \nu) = \int_0^1 \left| \left(\volterra \mu\right)^{-1}(t) - \left(\volterra \nu\right)^{-1}(t) \right|^2 dt ,
\end{equation}
where $(\volterra \mu)(s) = \int_{-\infty}^s \mu(x) dx$ denotes the cumulative distribution function and $(\volterra \mu)^{-1}(t) = \inf \{s : (\volterra \mu)(s) > t\}$ is the associated quantile function (or functional inverse of the cumulative distribution function).

The integration performed in the Radon transform typically smooths out differences between functions. One common approach in tomography to counteract this smoothing is to include a ramp-filter in Fourier space~\cite{Kak2001Principles}. We therefore additionally define a ramp-filtered version of the 3-D sliced Wasserstein distance (referred to as ramp-filtered sliced Wasserstein distance), as introduced in~\cite{ShiFast2026}:

\begin{equation}
\label{eq:ramp_filter}
\mu^s_{\theta,\varphi}(x) = h * (P_{\theta,\varphi})f(x) = \mathcal{F}_1^{-1} 
(|\xi|^2 (\mathcal{F}_3 f)(R_{\theta,\varphi}^\top \left( 0, 0, \xi)^\top \right),
\end{equation}
where $h$ is the ramp filter in the spatial domain, whose Fourier transform is $\mathcal{F}_1(h)(\xi) = |\xi|^2$. As previously described in~\cite{ShiFast2026}, applying the ramp filter does not enforce positivity of the line projections, which is required for probability measures. We therefore split each ramp-filtered projection into its positive and negative parts

\[
\mu^s_{\theta,\varphi} = \mu_{\theta,\varphi}^+ +  \mu_{\theta,\varphi}^-, \qquad \mu_{\theta,\varphi}^+ = \max(\mu_{\theta,\varphi},0), \quad \mu_{\theta,\varphi}^- = \min(\mu_{\theta,\varphi},0)
\]
 As noted in~\cite{ShiFast2026}, although the line projections are all mean-zero, their norms can be different. Specifically
 \[
 \int \mu_{\theta,\varphi}^+(x) + \mu_{\theta,\varphi}^-(x) dx = \int \nu_{\theta,\varphi}^+(x) + \nu_{\theta,\varphi}^-(x) dx = 0,
 \]
but it is not necessarily true that the norms of the signed projections are equal. Therefore, to follow the conditions required for the 1-D Wasserstein distance, the signed projections are additionally normalized to $1$, as in:
$$
\tilde{\mu}_{\theta,\varphi}^+ := \frac{\mu_{\theta,\varphi}^+}{\|\mu_{\theta,\varphi}^+\|_1}, \qquad \tilde{\mu}_{\theta,\varphi}^- := \frac{\mu_{\theta,\varphi}^-}{\|\mu_{\theta,\varphi}^-\|_1} .
$$
We then consider the 3-D ramp-filtered sliced 2-Wasserstein distance to be

\begin{equation}
\label{eq:ramp_filter_sliced_wasserstein}
d_{RFSW_2}^2 (f,g) = \frac{1}{4\pi} \int_0^{\pi}\int_0^{2\pi} d_{W_2}^2(\tilde{\mu}_{\theta,\varphi}^+ , \tilde{\nu}_{\theta,\varphi}^+) \sin \theta d\varphi d\theta + \frac{1}{4\pi} \int_0^{\pi}\int_0^{2\pi} d_{W_2}^2(\tilde{\mu}_{\theta,\varphi}^- , \tilde{\nu}_{\theta,\varphi}^-) \sin \theta d\varphi d\theta .
\end{equation}
Even with the additional normalization of the filtered projections,~\cref{eq:ramp_filter_sliced_wasserstein} provides a good measurement of the difference in the sense that $d_{RFSW_2}^2 (f,g) = 0$ if and only if $f = g$ almost everywhere.

\subsubsection{Step 3: computing alignment between inverse cumulative distribution functions}
In the last step, the distance is minimized over rotations $R \in \mathsf{SO}(3)$. In more detail, we choose the alignment angles $(\alpha_*,\beta_*,\gamma_*) \in [0,\pi]\times [0,2\pi]^2$ by minimizing the sliced Wasserstein distance between $f$ and $R_{\theta, \varphi}g$, i.e., by
\begin{equation}\label{eq:def_optimal_euler_angles}
(\alpha_*,\beta_*,\gamma_*) = \argmin_{(\alpha,\beta,\gamma) \in [0,\pi]\times [0,2\pi]^2} d_{SW_2}(f,R_{\alpha,\beta,\gamma}g).
\end{equation}

Our alignment algorithm hinges on the following result.

\begin{proposition}\label{prop:align}
    For any $(\alpha,\beta,\gamma) \in [0,\pi]\times [0,2\pi]^2$ and functions $f,g: \mathbb{R}^3 \rightarrow \mathbb{R}^+$ with $\int_{\mathbb{R}^3} f dx = 1 = \int_{\mathbb{R}^3} g dx$, we have
    \begin{equation}\label{eq:def_dsw2_theorem}
    d_{SW_2}(f,R_{\alpha,\beta,\gamma}g) = d_E(h_f, R_{\alpha,\beta,\gamma}h_g), 
\end{equation}
where $d_E(\cdot, \cdot)$ denotes the Euclidean distance and $$h_f(t,\theta,\varphi) = \tfrac{1}{t} \left(\volterra P_{\theta,\varphi} f\right)^{-1}(t), \qquad h_g(t,\theta,\varphi) = \tfrac{1}{t} \left(\volterra P_{\theta,\varphi} g\right)^{-1}(t) .$$
Similarly for the ramp-filtered sliced 2-Wasserstein distance
\begin{equation}\label{eq:ramp_filtered_alignment_2_terms}
    d_{RFSW_2}(f,R_{\alpha,\beta,\gamma}g)^2 = d_E(h^+_f, R_{\alpha,\beta,\gamma}h^+_g)^2 + d_E(h^-_f, R_{\alpha,\beta,\gamma}h^-_g)^2,
\end{equation}
where 
$$
h_f^\pm(t,\theta,\varphi) = \tfrac{1}{t} \left(\volterra \tilde{\mu}^\pm_{\theta,\varphi}\right)^{-1}(t), \qquad h_g^\pm(t,\theta,\varphi) = \tfrac{1}{t} \left(\volterra \tilde{\nu}^\pm_{\theta,\varphi}\right)^{-1}(t) .
$$

\end{proposition}
\begin{proof}
The proof relies on commuting a number of the introduced operators. We first note that
\begin{align}
\left(\volterra R_{\alpha,\beta,\gamma}P_{\theta,\varphi} g\right)(t) = \int_{-\infty}^t (R_{\alpha,\beta,\gamma}P_{\theta,\varphi} g)(z) dz = R_{\alpha,\beta,\gamma}\int_{-\infty}^t (P_{\theta,\varphi} g)(z) dz = \left(R_{\alpha,\beta,\gamma}\volterra P_{\theta,\varphi} g\right)(t).
\end{align}
It then also follows that
\begin{align}\label{eq:rot_icdf}
\left(\volterra R_{\alpha,\beta,\gamma}P_{\theta,\varphi} g\right)^{-1}(t) = \left(R_{\alpha,\beta,\gamma}\volterra P_{\theta,\varphi} g\right)^{-1}(t) =R_{\alpha,\beta,\gamma}\left(\volterra P_{\theta,\varphi} g\right)^{-1}(t).
\end{align}

We conclude by making the following chain of observations
\begin{align*}
d_{SW_2}(f,R_{\alpha,\beta,\gamma}g)^2  &= \frac{1}{4\pi} \int_0^{\pi}\int_0^{2\pi} d_{W_2}^2(P_{\theta,\varphi} f, P_{\theta,\varphi} R_{\alpha,\beta,\gamma}g) \sin \theta d\varphi d\theta && \text{ (by~\cref{eq:sliced_wasserstein})}\\
&= \frac{1}{4\pi} \int_0^{\pi}\int_0^{2\pi} d_{W_2}^2(P_{\theta,\varphi} f, R_{\alpha,\beta,\gamma}P_{\theta,\varphi} g) \sin \theta d\varphi d\theta && \text{ (by Obs.~\ref{obs:3})}\\
&= \frac{1}{4\pi} \int_0^{\pi}\int_0^{2\pi} \int_0^1 \left| \left(\volterra P_{\theta,\varphi} f\right)^{-1}(t) - \left(\volterra R_{\alpha,\beta,\gamma} P_{\theta,\varphi} g\right)^{-1}(t) \right|^2 dt\sin \theta d\varphi d\theta && \text{ (by~\cref{eq:1d_wasserstein})}\\
&= \frac{1}{4\pi} \int_0^{\pi}\int_0^{2\pi} \int_0^1 \left| \left(\volterra P_{\theta,\varphi} f\right)^{-1}(t) - \left(R_{\alpha,\beta,\gamma} \volterra P_{\theta,\varphi} g\right)^{-1}(t) \right|^2 dt\sin \theta d\varphi d\theta
&& \text{ (by~\cref{eq:rot_icdf})}\\
&= \frac{1}{4\pi} \int_0^{\pi}\int_0^{2\pi} \int_0^1 \left| \left(\volterra P_{\theta,\varphi} f\right)^{-1}(t) - R_{\alpha,\beta,\gamma}\left(\volterra P_{\theta,\varphi} g\right)^{-1}(t) \right|^2 dt\sin \theta d\varphi d\theta && \text{ (by~\cref{eq:rot_icdf})}\\
&= \frac{1}{4\pi} \int_0^{\pi}\int_0^{2\pi} \int_{0}^1 \left|\tfrac{1}{t} \left(\volterra P_{\theta,\varphi} f\right)^{-1}(t) - \tfrac{1}{t} R_{\alpha,\beta,\gamma}\left(\volterra P_{\theta,\varphi} g\right)^{-1}(t) \right|^2 t^2 \sin \theta dtd\varphi d\theta \\
&= d_E(h_f, R_{\alpha,\beta,\gamma}h_g)^2.
\end{align*}

For the ramp-filtered sliced Wasserstein distance, we have
$$
\volterra \left(h*(P_{\theta,\varphi}R_{\alpha,\beta,\gamma} f)\right)^\pm = \volterra \left(h*(R_{\alpha,\beta,\gamma} P_{\theta,\varphi}f)\right)^\pm
= \volterra R_{\alpha,\beta,\gamma}\left(h*( P_{\theta,\varphi}f)\right)^\pm = R_{\alpha,\beta,\gamma} \volterra \left(h*( P_{\theta,\varphi}f)\right)^\pm,
$$
where the first equality is Observation~\ref{obs:3}, and the second comes from the fact that $h$ is a radial function, so $R_{\alpha,\beta,\gamma}$ commutes with this radial convolution. The action of $R_{\alpha,\beta,\gamma}$ also commutes with extracting positive and negative parts, and the third equality is~\cref{eq:rot_icdf}. An analogous argument to the above chain of equalities then concludes the proof.
\end{proof}
In other words, we can compute the rotation minimizing the distance between $f$ and $g$ in the sliced 2-Wasserstein distance, by minimizing the Euclidean distance of the two helper functions $h_f$ and $h_g$. This can be done using any existing 3-D Euclidean alignment algorithm.
Similarly, for the ramp-filtered sliced 2-Wasserstein distance, we minimize the Euclidean distance between the two vector-valued helper functions \begin{equation} \label{eq:helper_vector_valued}
\begin{bmatrix}
    h_f^+ \\ h_f^-
\end{bmatrix} \quad \text{ and } \quad \begin{bmatrix}
    h_g^+ \\ h_g^-
\end{bmatrix}.
\end{equation}



\subsection{Discrete algorithm}
\label{sec:discrete_algorithm}
For discretized volumes, the steps in~\Cref{sec:cont} need to be performed on sampled data, given as voxel representations on discrete grids. This section explicitly writes out the discrete approximations of~\Cref{sec:cont} and shows how the steps can be performed with low computational cost. A visualization of the transforms used to compute the discretized sliced Wasserstein distance can be found in~\Cref{fig:transforms}.

\begin{figure}[!ht]
	\includegraphics[width=1\textwidth]{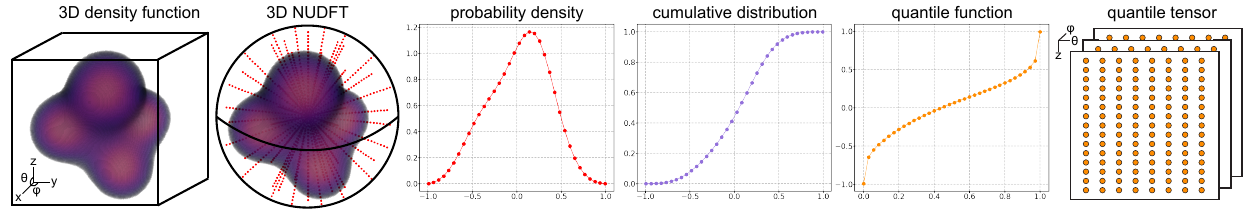}
	\centering
	\caption{Visualization of the transforms used to compute the 3-D sliced Wasserstein distance.}
	\label{fig:transforms}
\end{figure}

\subsubsection{Step 1: transforms}
We define a product grid of rotational angles in 3-D. Let the integers $n_\varphi$ and $n_\theta$ denote the number of grid-points in the two angular directions respectively and define a grid of rotational directions by
\begin{equation}
\varphi_j = \frac{2\pi j}{n_\varphi} , \quad j = 1,\ldots , n_\varphi, \qquad \text{ and } \qquad \theta_k = \frac{\pi k}{n_\theta} , \quad k = 1,\ldots , n_\theta ,
\end{equation}
and a Cartesian grid (with $n_x,n_y,n_z$ grid points in the $x$,$y$,$z$ directions, respectively), defined by
\begin{equation}\label{eq:defcartgrid}
x_p = \frac{2(p-1)}{n_x} - 1, \quad y_q = \frac{2(q-1)}{n_y} - 1, \quad z_i = \frac{2(i-1)}{n_z} - 1, \quad p = 1, \ldots , n_x, \quad q = 1, \ldots , n_y, \quad i = 1, \ldots , n_z. 
\end{equation}

Let $f,g : \mathbb{R}^3 \to \mathbb{R}^+$ be two probability measures which are normalized such that $\int_{\mathbb{R}^3} f(x,y,z) dx dy dz = \int_{\mathbb{R}^3} g(x, y, z) dx dy dz = 1$.
Given samples of $f$ and $g$ on the grid defined by $x_p,y_q,z_i$, we compute two tensors containing approximate projections of $f$ and $g$. Denoting these by $P_f,P_g\in \mathbb{R}^{N \times n_\varphi \times n_\theta}$, their entries are defined as discrete approximations to the projection tensor at $z_i, \theta_j, \varphi_k$. 
The approximated projection operator is given by performing a non-uniform fast Fourier transform (NUFFT)~\cite{Barnett2019Parallel, Barnett2021Aliasing} in Observation~\ref{obs:nufft}, followed by a fast Fourier transform. In detail, the projection operator is given by the following sequence of definitions:

\begin{enumerate}
\item For $\ell = 1, \ldots, n_z$, $j=1, \ldots , n_\varphi$, $k=1, \ldots , n_\theta$, define the NUFFT target points
\begin{align*}
\xi_{\ell j k} &= \frac{n_z}{2}(z_\ell \sin\theta_j\cos\varphi_k,z_\ell \sin\theta_j\sin\varphi_k,z_\ell \cos\theta_j),
\end{align*}
\item Compute the discrete Fourier transform using the NUFFT, as in
\begin{align*}
a_{\ell j k} &= \frac{8}{n_xn_yn_z}\sum_{p=1}^{n_x}\sum_{q=1}^{n_y} \sum_{i=1}^{n_z}f(x_p, y_q, z_i) e^{-\frac{\pi}{2} (x_p,y_q,z_i)\cdot \xi_{\ell j k}}, \quad \ell = 1, \ldots, n_z, \quad j = 1, \ldots, n_\theta, \quad k = 1, \ldots , n_\varphi.
\end{align*}
\item Compute the inverse 1-D discrete Fourier transform of the $a_{\ell j k}$, i.e., compute
\begin{align*}
(P_f)_{i,j,k} &= \frac{1}{2} \sum_{\ell=1}^{n_z} a_{\ell j k} e^{\frac{\pi}{2} n_z z_\ell z_i}, \quad {i = 1, \ldots , n_z,} \quad j = 1, \ldots , n_\theta, \quad k = 1, \ldots , n_\varphi.
\end{align*}
This step can be performed using a standard FFT, after noting that
$$ \sum_{\ell=1}^{n_z} a_{\ell j k} e^{\frac{\pi}{2} n_z z_\ell z_i} = \sum_{\ell=1}^{n_z} a_{\ell j k} e^{2\pi (\ell - n_z/2) (i -n_z/2)/n_z} .
$$
\end{enumerate}

The complexity of the NUFFT step performed with $\ell^1 - \ell^\infty$ relative error $\varepsilon$ is $\mathcal{O}(n_xn_yn_z\log(n_xn_yn_z) + n_zn_\varphi n_\theta |\log \varepsilon |^3)$~\cite{Barnett2019Parallel,Barnett2021Aliasing}. The complexity of the FFT step is $\mathcal{O}(n_\theta n_\varphi n_z \log n_z)$.

 For the ramp-filtered Wasserstein distance, we define a version of $a_{\ell j k}$ rescaled by the radial distance $|\xi_{\ell j k}|^2$, i.e.,
\begin{align}
a_{\ell j k}^s &= \frac{8}{n_xn_yn_z}\sum_{p=1}^{n_x}\sum_{q=1}^{n_y} \sum_{i=1}^{n_z}|\xi_{\ell j k}|^2f(x_p, y_q, z_i) e^{-\frac{\pi}{2} (x_p,y_q,z_i)\cdot \xi_{\ell j k}}, 
\end{align}
for each $\ell = 1, \ldots, n_z, \quad j = 1, \ldots, n_\theta, \quad k = 1, \ldots , n_\varphi.$ Subsequently, we define the signed and normalized Radon transforms
\begin{align*}
(P_f^+)_{i,j,k} &= \max\left(\frac{1}{2} \sum_{\ell=1}^{n_z} a_{\ell j k} e^{\frac{\pi}{2} n_z z_\ell z_i},0\right), \quad (P_f^-)_{i,j,k} = \min\left(\frac{1}{2} \sum_{\ell=1}^{n_z} a_{\ell j k} e^{\frac{\pi}{2} n_z z_\ell z_i},0\right) \\
(\tilde{P}_f^+)_{i,j,k} &= \frac{(P_f^+)_{i,j,k}}{\|P_f^{+}\|_1}, \quad (\tilde{P}_f^-)_{i,j,k} = \frac{(P_f^-)_{i,j,k}}{\|P_f^{-}\|_1}
\end{align*}
for all $i = 1, \ldots , n_z, \quad  j = 1, \ldots , n_\theta, \quad k = 1, \ldots , n_\varphi,$ where we define 
$$
\|P_f^{\pm}\|_1 = \sum_{i,j,k} w_{i,j,k}|(P_f^\pm)_{i,j,k}| \cdot \sin(\theta_k),
$$
for some quadrature weights $w_{i,j,k}$. We note that in our code implementation, the normalization is applied after transforming the signed projections into their respective discretized cdfs (described in the next section), which provides an equivalent normalization.


\subsubsection{Step 2: distances}
To compute the discretized 2-Wasserstein distance, we first convert the discretized projections $P_f$ and $P_g$ into their respective discretized cumulative distribution functions $\text{cdf}_{\theta_j, \varphi_k}f$ and $\text{cdf}_{\theta_j, \varphi_k}g$, which we now define. For fixed $j,k$, we approximate the integral defining cumulative distribution functions by quadrature on the points $z_i$, i.e., for an integer $m$, define
\begin{equation}\label{eq:def_cdf_quadrature}    
(\text{cdf}_{\theta_j, \varphi_k}f)(m) = \sum_{i=0}^m w_i (P_f)_{i,j,k},
\end{equation}
for some choice of quadrature weights $w_i$ for the nodes $z_i$.
Our implementation uses a trapezoid rule, but other choices are possible.

We define $(\text{cdf}_{\theta_j, \varphi_k}f)(s)$ for non-integer $s$ through piecewise linear interpolation.
Then, the discretized quantile function is
\begin{align}\label{eq:def_icdf_quadrature}
(\text{icdf}_{\theta_j, \varphi_k}f)(t) &:= \min\{s : (\text{cdf}_{\theta_j, \varphi_k}f)(s) > t \} .
\end{align}
The squared 2-Wasserstein distance $d_{W_2}^2 (P_{\theta_j,\varphi_k} f , P_{\theta_j,\varphi_k} g)$ is then approximated by
\begin{equation}\label{eq:def_Wass_quadrature}
 dW_{j,k}^2 := \sum_{i=1}^{n_t} w_i\left| (\text{icdf}_{\theta_j, \varphi_k}f)(t_i) - (\text{icdf}_{\theta_j, \varphi_k}g)(t_i) \right|^2,
\end{equation}
where $(t_i,w_i)$, $i= 1, \ldots , n_t$ are some choice of quadrature nodes and weights, respectively. Our code implements a choice of either Gauss-Legendre quarature or the composite Simpson's 1/3-rule.

We lastly approximate $d_{SW_2}^2 (f,g) $ in~\cref{eq:sliced_wasserstein} using $ dW_{j,k}^2$ and quadrature, i.e., compute
\begin{equation}
\label{eq:disc_sliced_wasserstein}
 dSW^2 := \frac{1}{4\pi} \sum_{j=1}^{n_\theta}\sum_{k=1}^{n_\varphi} w_{j,k} dW_{j,k}^2 \sin \theta_j,
\end{equation}
for some weights $w_{j,k}$. 

Specifically, our code implements composite Simpson's 1/3-rule in both the $\theta$ and $\varphi$ variables, but other choices are possible as well. For the ramp-filtered Wasserstein distance, we define $dSW^{\pm}$ as approximations to the two terms in~\cref{eq:ramp_filtered_alignment_2_terms} by replacing $P_f$ by $P_f^\pm$ in~\cref{eq:def_cdf_quadrature} and repeating equations~\cref{eq:def_icdf_quadrature} --~\cref{eq:disc_sliced_wasserstein} verbatim. We remark that this implementation is used for computing the discrete sliced 2-Wasserstein distance, but for rotational distance, we compute the optimizing rotation in~\cref{eq:def_optimal_euler_angles} using only~\cref{eq:def_icdf_quadrature} as described in the next section.

\subsubsection{Step 3: alignment and local refinement}
We use Proposition~\ref{prop:align} to find an optimal rotational alignment in the sliced Wasserstein distance, by aligning the two discretized helper tensors $h_f, h_g \in \mathbb{R}^{n_{t}\times n_\theta \times n_\varphi}$ defined by
\begin{align*}
(h_f)_{i,j,k} = \frac{1}{t_i}(\text{icdf}_{\theta_j, \varphi_k}f)(t_i), \quad (h_g)_{i,j,k} = \frac{1}{t_i}(\text{icdf}_{\theta_j, \varphi_k}g)(t_i), \quad i = 1, \ldots , n_t, \quad j = 1, \ldots , n_\theta, \quad k = 1, \ldots , n_\varphi,
\end{align*}
in the Euclidean distance. Any choice of existing Euclidean distance alignment algorithm can be employed, and we in particular use the algorithm from \cite{KovacsFast2002}. 

Briefly, this algorithm works by using a spherical harmonic representation to provide an efficient basis for rotational alignment. For a fixed radius $t_i$, the tensors of quantile functions $h_f$ and $h_g$ are expanded into spherical harmonic basis functions. The spherical harmonics coefficients of rotations of $h_f$ and $h_g$ are sums of Wigner-D matrices multiplied by the spherical harmonics coefficients of $h_f$ and $h_g$ themselves. This allows the rotational dependence of the alignment problem to be evaluated analytically and without loss of accuracy. The correlation between $h_g$ and all rotations of $h_f$ can therefore be computed by comparing their spherical harmonic coefficients weighted by Wigner-D matrices, enabling fast determination of the optimal rotation. In full detail, this algorithm defines a maximum bandlimit $L$ of the spherical harmonics and sets $n_\theta, n_\varphi = \mathcal{O}(L)$. This gives total complexity $\mathcal{O}\left((n_t+\log L)L^3\right)$, which therefore sets the overall complexity of our approach when $n_t, n_\theta, n_\varphi = \mathcal{O}(N)$. We refer the readers to \cite{KovacsFast2002} for more details.

For the case of the ramp-filtered sliced Wasserstein distance, we note that a small extension of the procedure of \cite{KovacsFast2002} allows for an alignment of the vector-valued functions in~\cref{eq:helper_vector_valued}. The only required change consists in computing the correlation between the vector-valued functions, after which the steps of the procedure can be repeated verbatim.

Lastly, we note that the approach of \cite{KovacsFast2002} computes a search of rotations on a regular grid of Euler angles. It therefore computes an approximation $(\alpha,\beta,\gamma)$ of the optimal Euler angles in~\cref{eq:def_optimal_euler_angles}, with accuracy determined by the resolution of the grid. This in turn improves with increasing $n_z, n_\theta, n_\phi$ and $n_t$, so a more accurate approximation can be obtained by increasing these parameters. However, we note that improved accuracy can also be achieved through a local refinement algorithm around the initially obtained angles $(\alpha,\beta,\gamma)$.

A simple local refinement algorithm consists in repeating the grid search of \cite{KovacsFast2002} on a grid localized around the obtained angles $(\alpha,\beta,\gamma)$. This can be done with small computational cost by storing the correlation between $h_f$ and $h_g$ computed above and replacing the last FFT step in \cite{KovacsFast2002} by a non-uniform FFT onto the localized grid. If desired, this can subsequently be repeated on increasingly fine localized grids. This approach is implemented in our code.

Alternatively, gradient-based local refinement can be used by computing the gradients of~\cref{eq:def_dsw2_theorem} and~\cref{eq:ramp_filtered_alignment_2_terms} with respect to $\alpha, \beta$ and $\gamma$. This can be done in closed-form after expanding $h_f$ and $h_g$ in spherical harmonics, which ensures that the rotational action becomes linear combinations of the spherical harmonic expansion coefficients with Wigner D-matrices.

\section{Stability estimates for the sliced Wasserstein distance in 3-D}
\label{sec:theory}
This section presents two stability estimates for the sliced Wasserstein distance.  Theorem~\ref{thm:translation} gives the exact value of the sliced Wasserstein distance between a function and a translation of itself. Theorem~\ref{thm:deformation_stability} gives an upper bound for the sliced Wasserstein between a function and a (general) deformation of itself. Moreover, we prove that the constant of this bound is sharp. In the context of alignment, after aligning two structures using the presented algorithm, these two results then give exact and lower bounds, respectively, on the degree of heterogeneity between the structures. For the sake of completeness, we present estimates valid for the sliced $p$-Wasserstein distance $d_{SW_p}$, for any $1\leq p < \infty$. This distance is defined replacing the exponent $2$ with $p$
throughout~\cref{eq:sliced_wasserstein}
and~\cref{eq:1d_wasserstein}.

The following result is a three-dimensional counterpart of
\cite[Theorem~3.2]{ShiFast2026}.
\begin{theorem}
\label{thm:translation}
Let $1\leq p<\infty$. Let $f:\mathbb{R}^3\to[0,\infty)$ be a
probability density function. For a translation
vector $s\in\mathbb{R}^3$, define
\[
    (T_s f)(x)=f(x-s).
\]
Then
\[
    d_{SW_p}(f,T_s f)
    =
    (p+1)^{-1/p}\|s\|.
\]
In particular, for $p=2$,
\[
    d_{SW_2}(f,T_s f)
    =
    \frac{1}{\sqrt{3}}\|s\|.
\]
\end{theorem}
The proof is given in \Cref{sec:shift_theory}. The next result extends the above translation analysis to general
spatially varying, mass-preserving deformations.
\begin{theorem}
\label{thm:deformation_stability}
Let $1\leq p<\infty$, and let
$f:\mathbb{R}^3\to[0,\infty)$ be a probability density function. Let $\Psi:\mathbb{R}^3\to\mathbb{R}^3$ be a
$C^1$ diffeomorphism, and let $f_\Psi$ denote the density
obtained by pushing $f$ forward through $\Psi$, namely
\[
    f_\Psi(y)
    =
    f\!\left(\Psi^{-1}(y)\right)
    \left|\det D\Psi^{-1}(y)\right|,
\]
where $D\Psi^{-1}$ denotes the Jacobian matrix of $\Psi^{-1}$.
Then
\[
    d_{SW_p}(f,f_\Psi)
    \leq
    (p+1)^{-1/p}
    \left(
    \int_{\mathbb{R}^3}
    \|\Psi(x)-x\|^p f(x)\,dx
    \right)^{1/p}.
\]
The constant $(p+1)^{-1/p}$ is sharp for every $1\leq p<\infty$.
In particular, for $p=2$,
\[
    d_{SW_2}(f,f_\Psi)
    \leq
    \frac{1}{\sqrt{3}}
    \left(
    \int_{\mathbb{R}^3}
    \|\Psi(x)-x\|^2 f(x)\,dx
    \right)^{1/2}.
\]
\end{theorem}
The proof is given in \Cref{sec:proof_deformation_stability}. We note that this result can be compared to \cite[Theorem~2.3]{LeebSliced2026}, which proves the bound
\begin{equation}\label{eq:def_max_bound}
        d_{SW_p}(f,f_\Psi)
    \leq \sup_{x\in\mathbb{R}^3}\|\Psi(x)-x\|.
\end{equation}
Since $\left(
    \int_{\mathbb{R}^3}
    \|\Psi(x)-x\|^p f(x)\,dx
    \right)^{1/p} \leq \sup_{x\in\mathbb{R}^3}\|\Psi(x)-x\| \left(
    \int_{\mathbb{R}^3}
    f(x)\,dx
    \right)^{1/p} = \sup_{x\in\mathbb{R}^3}\|\Psi(x)-x\|$ (because $f$ is a probability density function so is non-negative and integrates to $1$),~\Cref{thm:deformation_stability} improves~\cref{eq:def_max_bound}.

In addition to our direct proof, we can also prove~\Cref{thm:deformation_stability} using existing results.
Indeed, for general $p$, we apply the comparison
$d_{SW_p}^p\leq c_{3,p}d_{W_p}^p$ from
\cite[Proposition~5.1.3]{BonnotteUnidimensional2013}.
We evaluate its constant in three dimensions as
\[
    c_{3,p}
    =\frac13\int_{S^2}\|\theta\|_p^p\,d\sigma(\theta)
    =\frac12\int_{-1}^{1}|t|^p\,dt
    =\frac1{p+1},
\]
where $\sigma$ is the uniform probability measure on $S^2$.
Combining this comparison with the coupling induced by
$x\mapsto(x,\Psi(x))$, we obtain our deformation bound
with exactly the same factor $(p+1)^{-1/p}$.

\section{Numerical results}

In this section, we provide numerical results that demonstrate the properties of the sliced 2-Wasserstein distance for comparing voxelized 3-D functions, and for its application to rigid alignment over rotations.

\subsection{Timing results}
First we show the computational runtime to transform the voxelized 3-D data into the quantile tensors used for the sliced 2-Wasserstein and ramp-filtered sliced 2-Wasserstein distance. This step reflects the only precompute time required to transform the data prior to using other voxel-based alignment algorithms. As described in~\Cref{sec:discrete_algorithm}, the complexity of the NUFFT step performed with $\ell^1 - \ell^\infty$ relative error $\varepsilon$ is $\mathcal{O}(n_xn_yn_z\log(n_xn_yn_z) + n_zn_\varphi n_\theta |\log \varepsilon |^3)$. Specifically, we use the Flatiron Institute NUFFT implementation in our code~\cite{Barnett2019Parallel, Barnett2021Aliasing}. The complexity of the subsequent FFT step is $\mathcal{O}(n_\theta n_\varphi n_z \log n_z)$. We show timing results for increasing number of voxels and report the average time for $n=5$ trials in~\Cref{tab:timing}. All timing results are performed using a 2.6 GHz Intel Skylake processor and 32 GB of memory. These results show that the linear transforms required to compute the quantile tensor grow approximately as $\mathcal{O}(N^3 \log N)$. We show timing results for the full alignment algorithm compared to other methods in~\Cref{sec:cryo}.

\begin{table}[h]
\centering
\caption{Quantile tensor transform runtime. Computational runtime for transforming voxel data into the quantile tensor used to compute the the sliced 2-Wasserstein and ramp-filtered sliced 2-Wasserstein distance for increasing $N \times N \times N$ volumes. Values are reported in seconds as the mean and standard deviation of $n=5$ trials.}
\resizebox{\textwidth}{!}{%
\begin{tabular}{cccccc}
\hline
Metric & 32 & 64 & 128 & 256 & 512 \\
\hline
Sliced 2-Wasserstein & 0.023 $\pm$ 0.002 & 0.118 $\pm$ 0.05 & 0.627 $\pm$ 0.03 & 4.796 $\pm$ 0.015 & 39.879 $\pm$ 0.483 \\
Ramp-Filtered Sliced 2-Wasserstein & 0.028 $\pm$ 0.001 & 0.127 $\pm$ 0.001 & 0.828 $\pm$ 0.003 & 6.342 $\pm$ 0.028 & 52.249 $\pm$ 0.619 \\
\hline
\end{tabular}
}
\label{tab:timing}
\end{table}

\subsection{Translation equivariance}
\label{sec:shift_results}

Next, we demonstrate the translational equivariance property of the sliced Wasserstein distance in 3-D. As shown in~\Cref{thm:translation}, given a 3-D function $f(x)$ and its shifted version $f(x - s)$, the sliced Wasserstein distance will change proportionally to the magnitude of the shift. To demonstrate this, we generate a simple 3-D Gaussian as $\exp(-(x^2 + y^2 + z^2)/2\sigma^2)$, with $\sigma=5$, on a $N \times N \times N = 128 \times 128 \times 128 $ voxel grid. We then shift the Gaussian in voxel increments along one axis, corresponding to a distance of $0.015625$ on our defined scale of $[-1, 1]$, and compute the squared Euclidean, sliced 2-Wasserstein and ramp-filtered sliced 2-Wasserstein distance with the original (non-shifted) Gaussian. In~\Cref{fig:shift}, we see a linear response from the optimal transport metrics but not for the Euclidean distance. This result demonstrates the favorable translational equivariance property of the sliced Wasserstein distance.


\begin{figure}[!ht]
	\includegraphics[width=1\textwidth]{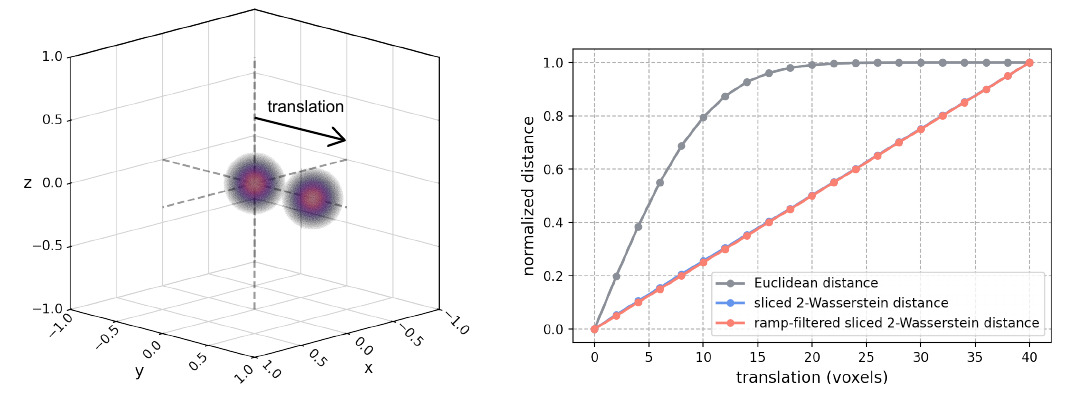}
	\centering
	\caption{Translation equivariance of the sliced 2-Wasserstein distance in 3-D. A reference 3-D volume (standard 3-D Gaussian, with $\sigma=5$), is shifted in voxel increments and the distance is computed in Euclidean, sliced 2-Wasserstein and ramp-filtered sliced 2-Wasserstein distance between the reference and shifted volume. Each metric is normalized by its maximum value for visualization.}
	\label{fig:shift}
\end{figure}

\subsection{ModelNet10 data alignment}

Next we test rotational alignment in sliced Wasserstein distance using the ModelNet10~\cite{ZhirongWu3D2015} dataset. This dataset consists of computer-aided design models from $10$ different categories (bathtub, bed, chair, desk, dresser, monitor, night stand, sofa, table, toilet). The different design of the models in each category provides a challenging dataset for 3-D alignment algorithms. Each model is represented by a mesh that is voxelized into a $N \times N \times N = 64 \times 64 \times 64$ binary occupancy grid using the voxelization implementation in the Trimesh Python library. We then apply a random rotation to each model.

For each model category, we compare the alignment results using the Euclidean, sliced 2-Wasserstein and ramp-filtered sliced 2-Wasserstein distance. The models in each category are aligned to a reference model, which was selected as the model that minimizes the Euclidean distance to the average model of that category. Therefore, we expect the alignment results to favor alignment in Euclidean distance. 

For a qualitative view of the alignment results, we display the first $11$ models in the chair category aligned under each metric in~\Cref{fig:modelnet_align}. We then summarize the alignment results for all models in the ModelNet10 test set, where the number of models per category ranges from $n=50$ to $n=100$. To quantify the alignment results, we compute the geodesic distance between the ground truth rotation matrices and the estimated rotation matrices $R_1, R_2 \in \mathsf{SO}(3)$, as defined by

\begin{equation}
    \label{eq:geodesic_angle_distance}
    \Theta(R_1,R_2) = \cos^{-1} \left( \frac{tr(R_1^\top R_2) - 1}{2} \right) .
\end{equation}
Specifically, because the models are considered to be aligned before the random rotation is applied, we measure how well the estimated alignment rotation matches the initial, random (ground truth) rotation. We then plot the cumulative number of models aligned within increasing radians from the ground truth rotation. We additionally note that some categories may contain symmetric models (e.g., table), which is not accounted for in the geodesic distance. Overall, these results show that the sliced Wasserstein metrics are either comparable, or greatly outperform the Euclidean distance for rotational alignment of the ModelNet10 data.

\begin{figure}[!ht]
	\includegraphics[width=1\textwidth]{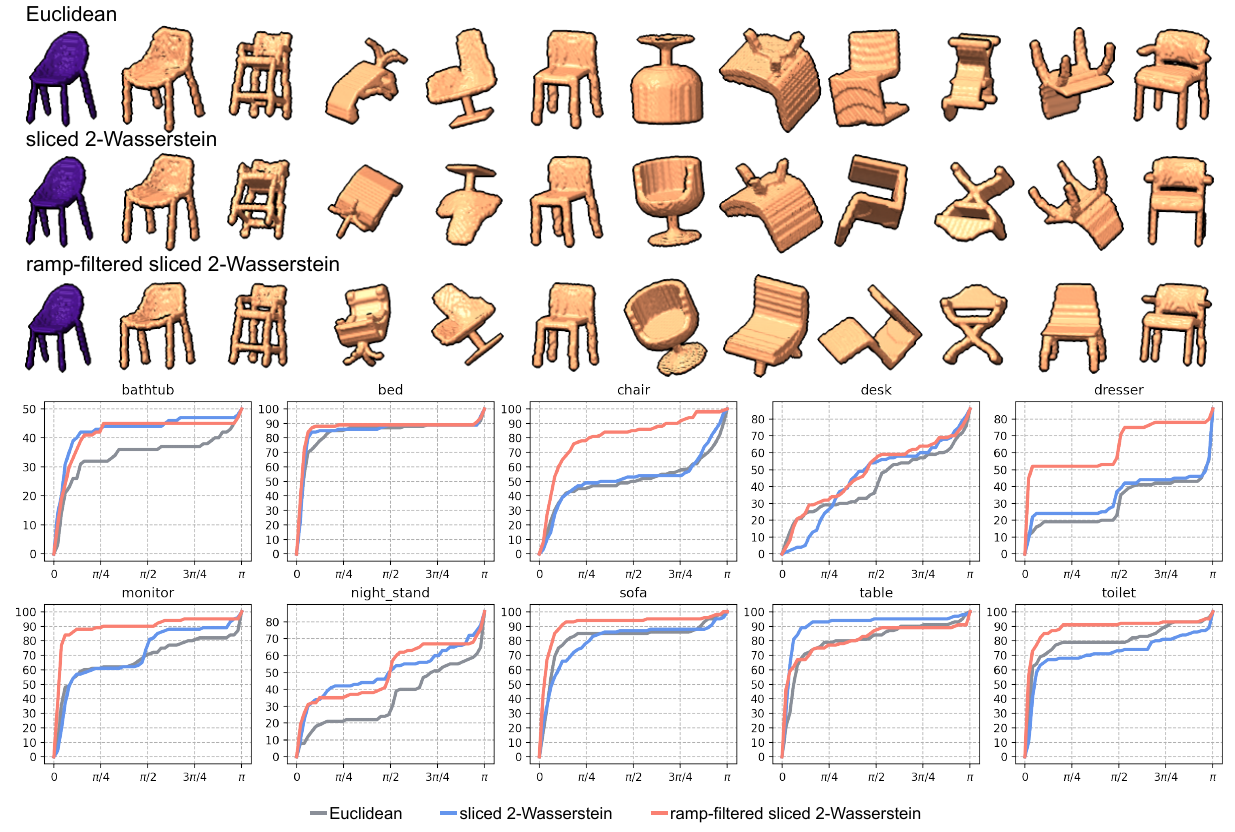}
	\centering
	\caption{Rotational alignment of the ModelNet10 data using different metrics. Models are discretized on a $N \times N \times N = 64 \times 64 \times 64$ voxel grid and are aligned to a reference model in Euclidean, sliced 2-Wasserstein and ramp-filtered sliced 2-Wasserstein distance. (Top) Example alignment of the first $11$ models from the chair category. The reference model is shown in purple and the aligned models are shown in gold. (Bottom) Summary of alignment results for $10$ different categories. Each dataset contains $n \approx 100$ models, and the plots show the number of correctly aligned models as a function of increasing radians under the geodesic distance in~\cref{eq:geodesic_angle_distance}.}
	\label{fig:modelnet_align}
\end{figure}

\subsection{Alignment of 3-D cryo-EM data}
\label{sec:cryo}

A common task in cryo-EM pipelines is to align 3-D models~\cite{PettersenUCSF2021, SingerAlignment2024, RiahiAlignOT2023, HarpazThreedimensional2023, HeAccurate2024}. Cryo-EM data often contains two types of heterogeneity, which can challenge conventional alignment procedures. They are; 1) compositional heterogeneity, where the data contains different underlying structures, and 2) conformational heterogeneity, where the data contains the same underlying structure but with deformations. In~\Cref{sec:homogeneous}, we first benchmark our alignment algorithm compared to other OT alignment methods using cryo-EM data without heterogeneity. Then, in~\Cref{sec:heterogeneous}, we demonstrate the results of our approach on heterogeneous cryo-EM data.

\subsubsection{Homogeneous Alignment}
\label{sec:homogeneous}

First we benchmark our alignment algorithm using 3-D models from published cryo-EM data without heterogeneity. We compare our algorithm to two other OT based methods including BOTalign~\cite{SingerAlignment2024} and AlignOT~\cite{RiahiAlignOT2023}. We test the alignment algorithms for $5$ different cryo-EM volumes from the Electron Microscopy Data Bank (EMDB)~\cite{ThewwPDBConsortium2024EMDB} ranging in molecular weight from $\approx 170 - 2000$ kDa, and resolution from $\approx 2.5 - 20 \text{\AA}$, as shown in~\Cref{fig:cryo_vols}. In this section, we refer to our method as Fast Sliced Wasserstein 3-D (FSW3D), and specifically use the ramp-filtered sliced 2-Wasserstein version of the algorithm for our results.

\begin{figure}[!ht]
	\includegraphics[width=1\textwidth]{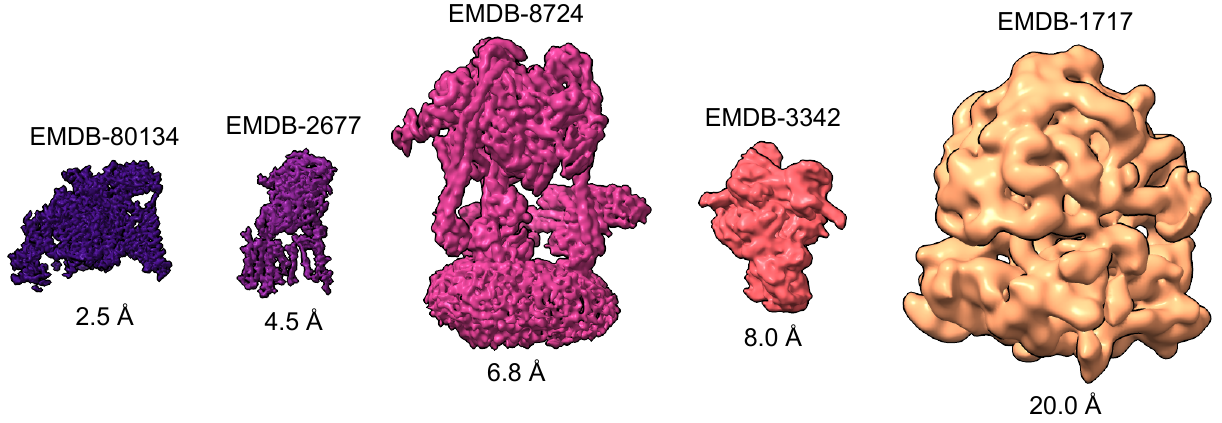}
	\centering
	\caption{Cryo-EM volumes from the EMDB~\cite{ThewwPDBConsortium2024EMDB} used for the homogeneous alignment experiments.}
	\label{fig:cryo_vols}
\end{figure}
In the first experiment, each volume is downsampled to $N \times N \times N = 64 \times 64 \times 64$ using the cryo-EM software ASPIRE~\cite{GarrettWright2026ComputationalCryoEM}. To compare the different algorithms, we compute the geodesic distance~\cref{eq:geodesic_angle_distance} between the estimated rotation alignment matrix and the ground truth rotation matrix for $n=50$ trials, and report the mean and standard deviation in degrees, shown in~\Cref{tab:alignment_accuracy}. Euler angles for the ground truth rotations are selected uniformly and randomly as integers in $\{0, \ldots ,359\}$. For BOTalign, we use the hyperparameter $\{T=150\}$ for the total number of iterations. For AlignOT, we use the cryo-EM volume threshold reported in the EMDB (required by AlignOT) and the hyperparameters $\{n=500, \alpha=0.1, \epsilon=100, L=100\}$ which correspond to the number of points, the learning rate parameter, the regularization parameter and the maximum iterations, respectively. We additionally compare our method and BOTalign with and without their respective local refinement procedure. We also note that for AlignOT, alignment results tend to improve when the initialization is good~\cite{RiahiAlignOT2023} (i.e., the rotated volume is close to the original), however that is not the setting we are testing here. Our results show that, in general, FSW3D with local refinement is the best performing algorithm. Furthermore, we note that because our algorithm is deterministic and does not require hyperparameters or good initialization, the variance is lower than the other methods. 

\begin{table}[!ht]
\centering
\caption{Alignment accuracy results. The geodesic distance is computed between the estimated rotation and ground truth rotation for $n=50$ random rotations of $5$ different cryo-EM volumes from the EMDB~\cite{ThewwPDBConsortium2024EMDB}, and is reported as the mean and standard deviation. The volumes are first downsampled to $N \times N \times N = 64 \times 64 \times 64$. Values are reported in degrees.}
\begin{tabular}{lccccc}
\hline
 & EMD-80134 & EMD-2677 & EMDB-8724 & EMD-3342 & EMD-1717 \\
\hline

FSW3D            & $2.12 \pm 0.77$ & $2.12 \pm 0.77$ & $2.14 \pm 0.78$ & $2.13 \pm 0.76$ & $2.12 \pm 0.77$ \\

FSW3D + local    & $0.09 \pm 0.04$ & $0.13 \pm 0.05$ & $0.21 \pm 0.09$ & $0.22 \pm 0.1$ & $0.08 \pm 0.04$ \\

BOTalign         & $3.52 \pm 3.11$ & $9.54 \pm 19.18$ & $16.4 \pm 35.29$ & $2.65 \pm 1.19$ & $3.51 \pm 3.08$ \\

BOTalign + local & $0.36 \pm 0.48$ & $0.21 \pm 0.09$ & $12.11 \pm 35.94$ & $0.22 \pm 0.1$ & $0.07 \pm 0.03$ \\

AlignOT          & $78.04 \pm 67.81$ & $30.47 \pm 43.86$ & $124.67 \pm 44.92$ & $107.31 \pm 56.86$ & $119.77 \pm 36.97$ \\

\hline
\end{tabular}
\label{tab:alignment_accuracy}
\end{table}
Next we report the runtimes for each algorithm with volumes of size $N = \{32, 64, 128\}$ in~\Cref{tab:alignment_timing}. The table shows the mean and standard deviation of $n=50$ trials, reported in seconds. All results in~\Cref{sec:homogeneous} were done on a computer a 2.6 GHz Intel Skylake processor and 32 GB of memory. The reported computational complexity of each method is $\mathcal{O}(N^4)$ for FSW3D, $\mathcal{O}(T N^3 \log N)$ for BOTalign and $\mathcal{O}(n^2 L \log n \epsilon^{-3})$ for AlignOT, where $n$ is the number of points in the point cloud and $L$ is the number of iterations. We note that for AlignOT, the complexity is dependent on the user defined number of points rather than the input number of voxels. However, because no guidance is given for selecting this parameter, we keep it fixed at $n=500$ rather than scaling it linearly or cubically with the number of voxels. The more direct comparison between FSW3D and BOTalign shows that FSW3D has faster runtime at all sizes, even though the asymptotic complexity is worse. This section demonstrates that FSW3D has higher accuracy and faster runtimes among other OT based alignment algorithms. 

\begin{table}[!ht]
\centering
\caption{Timing results for EMDB-1717 at increasing number of voxels, $N = \{32, 64, 128\}$. Reported values are the mean and standard deviation of $n=50$ trials, and are shown in seconds.}
\begin{tabular}{lccccc}
\hline
 & FSW3D & FSW3D + local & BOTalign & BOTalign + local & AlignOT \\
\hline

$N=32$   & $0.23 \pm 0.01$ & $0.38 \pm 0.07$ & $6.57 \pm 0.14$ & $12.25 \pm 1.9$ & $4.83 \pm 0.04$ \\

$N=64$   & $0.95 \pm 0.02$ & $1.47 \pm 0.19$ & $37.35 \pm 0.55$ & $42.9 \pm 1.27$ & $4.81 \pm 0.05$ \\

$N=128$  & $7.58 \pm 0.08$ & $10.94 \pm 1.38$ & $128.47 \pm 0.87$ & $142.2 \pm 2.9$ & $6.17 \pm 0.25$ \\

\hline
\end{tabular}
\label{tab:alignment_timing}
\end{table}

\subsubsection{Heterogeneous Alignment}
\label{sec:heterogeneous}

Next we test alignment results of our algorithm using heterogeneous cryo-EM data. Specifically, we focus on the case of conformational heterogeneity. We use a dataset from~\cite{ZhongCryoDRGN2021} that contains an $N \times N \times N = 128 \times 128 \times 128$ volume with simulated motion represented by $n=50$ discrete states. In~\Cref{fig:cryo}a, the models are aligned to the initial (reference) volume to show the axis of motion in the data. We then compute the distance, not over rotations, between the initial volume and each of the subsequent volumes in motion using the Euclidean, sliced 2-Wasserstein and ramp-filtered sliced 2-Wasserstein distance. In this example of simple motion, the Wasserstein metrics demonstrate the favorable property of stability under deformation. However, this property will not always hold for more complicated motion. Nonetheless, this result suggests that, given a set of aligned volumes, the Wasserstein metrics can provide a more descriptive measure of similarity.

Next, in~\Cref{fig:cryo}b, we compare the rotational alignment between the initial and final volume using each metric. Notably, the alignment from all metrics yield different orientations. In particular, the Euclidean distance appears to align the `head' of the volume to be aligned correctly, but the `leg' no longer overlaps with the reference. Meanwhile, the Wasserstein metrics appear to keep all parts of the volumes closer together, but with some parts aligned in different orientations (e.g. the `head' is rotated compared to the reference). This result demonstrates that the different metrics yield distinct alignment results for heterogeneous data and that, depending on the task, one metric might be preferred over the other.

\begin{figure}[!ht]
	\includegraphics[width=1\textwidth]{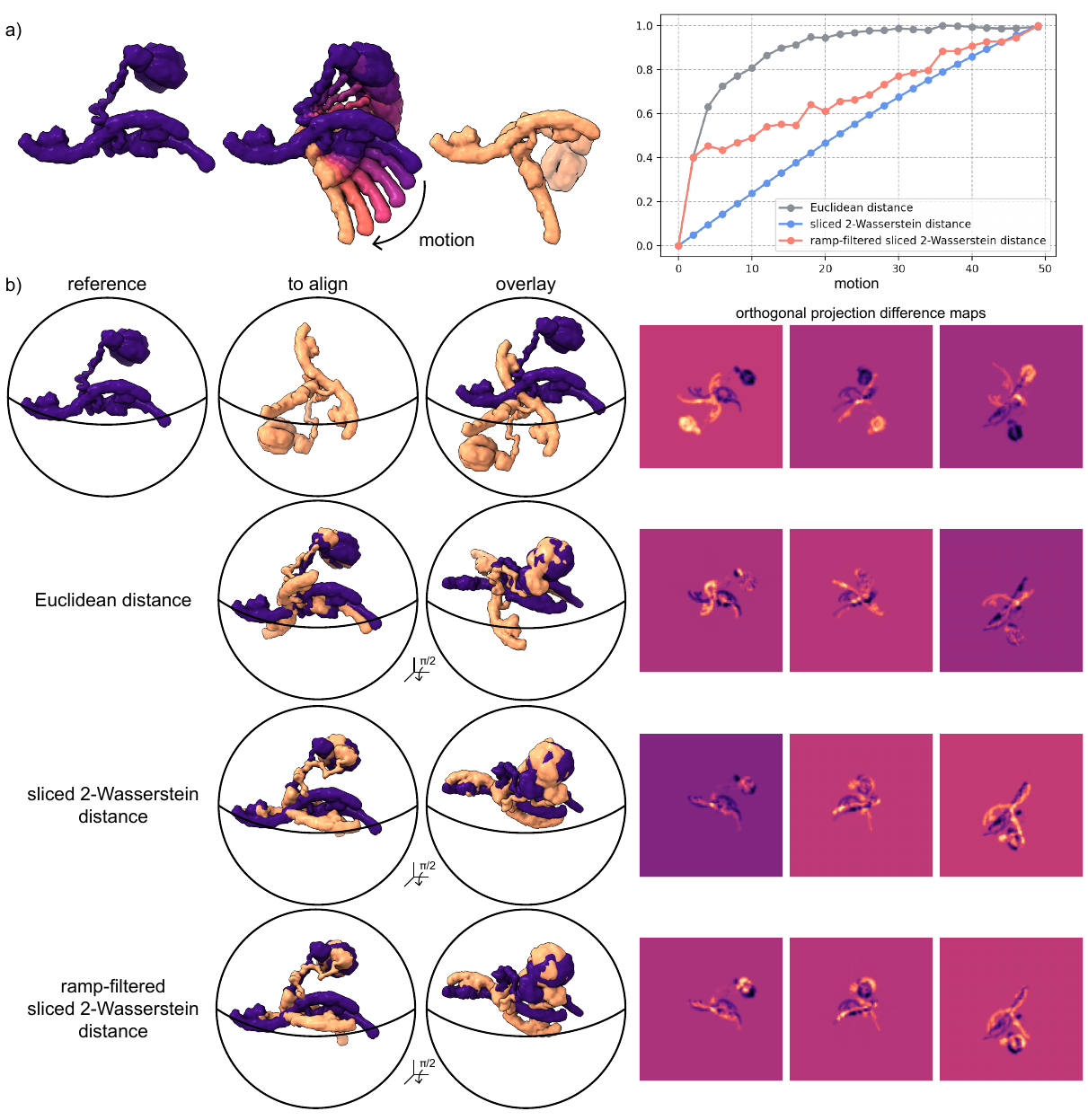}
	\centering
	\caption{Distance and alignment of heterogeneous cryo-EM data. (a) Cryo-EM volumes of size $N \times N \times N = 128\times 128\times 128$ with simulated motion (see~\cite{ZhongCryoDRGN2021} for more details about the data). The distances are computed between the initial volume and the subsequent volumes with increasing state of motion. (b) Rotational alignment of the final volume to the initial volume under each metric. Orthogonal projection plots show the difference between the initial and final volume after rotational alignment, projected onto each orthogonal plane. Alignment in both sliced Wasserstein distances gives similar results which are distinct from the alignment in Euclidean distance.}
	\label{fig:cryo}
\end{figure}

\section{Discussion}

In this work, we develop a framework for fast alignment over $\mathsf{SO}(3)$ in the sliced 2-Wasserstein distance. Additionally, we provide theoretical results on the stability of sliced 2-Wasserstein distance in 3-D. This work extends the applicability of optimal transport and sliced probability metrics to the 3-D case, in particular for high-dimensional data that requires alignment over a large grid of rotations. The ability to rapidly compute optimal transport based metrics for large, voxelized data will allow for an alternative way to compare objects with distinct 3-D structures. We also note that this algorithm is readily implementable on GPUs.

Several important areas remain to be addressed to increase the applicability of this method. For example, this analysis assumes that the input volumes are strictly positive and have the same initial mass. However it is common to have mean-subtracted or noisy data. Another interesting topic to explore comes from observing that there is no true alignment between non-identical functions. To analyze this further, one could test the downstream effects of algorithms that first require 3-D alignment by comparing the results from both the Euclidean distance and sliced Wasserstein distance based alignments, similar to the work in~\cite{WoollardImproving2025}. Other algorithmic choices worth exploring include the choice of grid points, frequency marching schemes~\cite{KruseFast2026}, multigrid schemes, numerical integration methods, interpolation schemes, and connections to related invariant metrics~\cite{BeckmannGeneralizations2026}.




\clearpage
\newpage

\appendix
\section{Proofs}
\subsection{Proof of Theorem~\ref{thm:translation}}
\label{sec:shift_theory}



\begin{proof}
Define
\[
    u_{\theta,\varphi}:=R_{\theta,\varphi}^\top e_z
    =
    (\sin\theta\cos\varphi,\sin\theta\sin\varphi,\cos\theta)^\top .
\]
Thus $u_{\theta,\varphi}$ is the projection direction corresponding to
$P_{\theta,\varphi}$. Write $g(x) = f(x-s)$. For each $z\in\mathbb{R}$,
\[
\begin{aligned}
(P_{\theta,\varphi}g)(z)
&=
\int_{\mathbb{R}^2}
g\!\left(R_{\theta,\varphi}^\top (x,y,z)^\top\right)\,dx\,dy \\
&=
\int_{\mathbb{R}^2}
f\!\left(R_{\theta,\varphi}^\top (x,y,z)^\top - s\right)\,dx\,dy \\
&=
\int_{\mathbb{R}^2}
f\!\left(R_{\theta,\varphi}^\top\big((x,y,z)^\top - R_{\theta,\varphi}s\big)\right)\,dx\,dy .
\end{aligned}
\]
Writing $R_{\theta,\varphi}s=((R_{\theta,\varphi}s)_1,(R_{\theta,\varphi}s)_2,
(R_{\theta,\varphi}s)_3)^\top$ and changing variables
\[
    x'=x-(R_{\theta,\varphi}s)_1,\qquad
    y'=y-(R_{\theta,\varphi}s)_2,
\]
we obtain
\[
\begin{aligned}
(P_{\theta,\varphi}g)(z)
&=
\int_{\mathbb{R}^2}
f\!\left(R_{\theta,\varphi}^\top
(x',y',z-(R_{\theta,\varphi}s)_3)^\top\right)
\,dx'\,dy' \\
&=
(P_{\theta,\varphi}f)\!\left(z-(R_{\theta,\varphi}s)_3\right).
\end{aligned}
\]
Since
\[
    (R_{\theta,\varphi}s)_3
    =
    e_z^\top R_{\theta,\varphi}s
    =
    (R_{\theta,\varphi}^\top e_z)^\top s
    =
    u_{\theta,\varphi}\cdot s,
\]
we have
\[
    P_{\theta,\varphi}g
    =
    \tau_{u_{\theta,\varphi}\cdot s} P_{\theta,\varphi}f,
\]
where $\tau_a \mu(z)=\mu(z-a)$ denotes translation by $a$ in one dimension.

For any one-dimensional probability density $\mu$ and any $a\in\mathbb{R}$,
the quantile function of $\tau_a\mu$ is shifted by $a$:
\[
    (\mathcal{V}{\tau_a\mu})^{-1}(t)
    =
    (\mathcal{V}\mu)^{-1}(t)+a .
\]
Therefore,
\[
    d_{W_p}^p(\mu,\tau_a\mu)
    =
    \int_0^1 |a|^p\,dt
    =
    |a|^p .
\]
Applying this with
$a=u_{\theta,\varphi}\cdot s$ gives
\[
    d_{W_p}^p(P_{\theta,\varphi}f,P_{\theta,\varphi}g)
    =
    |u_{\theta,\varphi}\cdot s|^p .
\]
Hence
\[
\begin{aligned}
d_{SW_p}^p(f,T_s f)
&=
\frac{1}{4\pi}
\int_0^\pi
\int_0^{2\pi}
|u_{\theta,\varphi}\cdot s|^p
\sin\theta\,d\varphi\,d\theta .
\end{aligned}
\]
By rotational invariance of the surface measure on $S^2$, the integral depends
on $s$ only through $\|s\|$. Thus we may assume without loss of generality that
$s=\|s\|e_z$. Then
\[
    u_{\theta,\varphi}\cdot s=\|s\|\cos\theta,
\]
and so
\[
\begin{aligned}
d_{SW_p}^p(f,T_s f)
&=
\frac{\|s\|^p}{4\pi}
\int_0^\pi
\int_0^{2\pi}
|\cos\theta|^p\sin\theta\,d\varphi\,d\theta \\
&=
\frac{\|s\|^p}{2}
\int_0^\pi
|\cos\theta|^p\sin\theta\,d\theta \\
&=
\frac{\|s\|^p}{2}
\int_{-1}^{1}
|t|^p\,dt \\
&=
\frac{\|s\|^p}{2}\cdot \frac{2}{p+1}
=
\frac{1}{p+1}\|s\|^p .
\end{aligned}
\]
Taking the $p$-th root completes the proof.
\end{proof}

\subsection{Proof of Theorem~\ref{thm:deformation_stability}}\label{sec:proof_deformation_stability}
\begin{proof}
For $(\theta,\varphi)\in[0,\pi]\times[0,2\pi)$, let
\[
    u_{\theta,\varphi}
    =
    (\sin\theta\cos\varphi,
     \sin\theta\sin\varphi,
     \cos\theta)^\top
\]
be the corresponding projection direction. 
Let $X$ be an $\mathbb{R}^3$-valued random vector with density
$f$, and set $Y=\Psi(X)$, so that $Y$ has density $f_\Psi$.
Fix $(\theta,\varphi)$ and write $u=u_{\theta,\varphi}$.
For every Borel set $A\subseteq\mathbb{R}$,
\[
    \mathbb{P}(u\cdot X\in A)
    =
    \int_{\{x:\,u\cdot x\in A\}} f(x)\,dx
    =
    \int_A (P_{\theta,\varphi}f)(t)\,dt,
\]
where the last equality follows by integrating $f$ over
the planes perpendicular to $u$. Thus, $u\cdot X$ has density
$P_{\theta,\varphi}f$. The same argument shows that
$u\cdot Y$ has density $P_{\theta,\varphi}f_\Psi$.
Consequently, the joint law of
\[
    (u\cdot X,\;u\cdot Y)
    =
    (u\cdot X,\;u\cdot\Psi(X))
\]
is a coupling of these two projected densities. Therefore,
by the Kantorovich formulation of the Wasserstein distance
\cite{VillaniOptimal2009, PeyreComputational2019},
\[
\begin{aligned}
d_{W_p}^p
\left(
P_{\theta,\varphi}f,
P_{\theta,\varphi}f_\Psi
\right)
&\leq
\mathbb{E}\left[
    \left|u\cdot Y-u\cdot X\right|^p
\right] \\
&=
\int_{\mathbb{R}^3}
\left|
u_{\theta,\varphi}\cdot
\big(\Psi(x)-x\big)
\right|^p f(x)\,dx.
\end{aligned}
\]
Averaging this inequality over all projection directions and applying
Tonelli's theorem gives
\[
\begin{aligned}
d_{SW_p}^p(f,f_\Psi)
&\leq
\int_{\mathbb{R}^3}
\left[
\frac{1}{4\pi}
\int_0^\pi
\int_0^{2\pi}
\left|
u_{\theta,\varphi}\cdot
\big(\Psi(x)-x\big)
\right|^p
\sin\theta\,d\varphi\,d\theta
\right]
f(x)\,dx .
\end{aligned}
\]
For every $v\in\mathbb{R}^3$, rotational invariance of the uniform
surface measure on $S^2$ implies
\[
\begin{aligned}
\frac{1}{4\pi}
\int_0^\pi
\int_0^{2\pi}
|u_{\theta,\varphi}\cdot v|^p
\sin\theta\,d\varphi\,d\theta
&=
\frac{\|v\|^p}{2}
\int_{-1}^{1}|t|^p\,dt 
&=
\frac{1}{p+1}\|v\|^p .
\end{aligned}
\]
Applying this identity with $v=\Psi(x)-x$ yields
\begin{equation}\label{eq:def_bound_proof}
    d_{SW_p}^p(f,f_\Psi)
    \leq
    \frac{1}{p+1}
    \int_{\mathbb{R}^3}
    \|\Psi(x)-x\|^p f(x)\,dx .
\end{equation}
Taking the $p$-th root proves the stated bound. 

To prove sharpness of the bound, we choose $\Psi_s(x)=x+s$. The push-forward density is
\[
    f_{\Psi_s}(x)=f(x-s)=T_s f(x),
\]
and its density-weighted displacement is exactly
\[
    \left(
    \int_{\mathbb{R}^3}
    \|\Psi_s(x)-x\|^p f(x)\,dx
    \right)^{1/p}
    =
    \|s\|.
\]
~\cref{eq:def_bound_proof} therefore gives
\[
    d_{SW_p}(f,T_s f)
    \leq
    (p+1)^{-1/p}\|s\|.
\]
Theorem~\ref{thm:translation} shows that equality always holds. Thus, the translation theorem is the exact equality case of the general deformation bound and proves that its constant cannot in general be improved. For a general nonlinear deformation, the coupling used in the proof need not be optimal after every one-dimensional projection, so the inequality can be strict.
\end{proof}

\printbibliography

@article{Barnett2021Aliasing,
  title = {Aliasing Error of the Exp  ( {$\beta$} 1 - z 2 ) Kernel in the Nonuniform Fast {{Fourier}} Transform},
  author = {Barnett, Alex H.},
  year = 2021,
  month = mar,
  journal = {Applied and Computational Harmonic Analysis},
  volume = {51},
  pages = {1--16},
  issn = {10635203},
  doi = {10.1016/j.acha.2020.10.002},
  url = {https://linkinghub.elsevier.com/retrieve/pii/S1063520320300725}
}

@article{Barnett2019Parallel,
  title = {A {{Parallel Nonuniform Fast Fourier Transform Library Based}} on an ``{{Exponential}} of {{Semicircle}}" {{Kernel}}},
  author = {Barnett, Alexander H. and Magland, Jeremy and Af Klinteberg, Ludvig},
  year = 2019,
  month = jan,
  journal = {SIAM Journal on Scientific Computing},
  volume = {41},
  number = {5},
  pages = {C479-C504},
  issn = {1064-8275, 1095-7197},
  doi = {10.1137/18M120885X},
  url = {https://epubs.siam.org/doi/10.1137/18M120885X}
}

@article{BeckmannGeneralizations2026,
  title = {Generalizations of the {{Normalized Radon Cumulative Distribution Transform}} for {{Limited Data Recognition}}},
  author = {Beckmann, Matthias and Beinert, Robert and Bresch, Jonas},
  year = 2026,
  month = aug,
  journal = {Journal of Mathematical Imaging and Vision},
  volume = {68},
  number = {4},
  pages = {45},
  issn = {0924-9907, 1573-7683},
  doi = {10.1007/s10851-026-01327-4},
  url = {https://link.springer.com/10.1007/s10851-026-01327-4}
}

@article{BendorySingleParticle2020,
  title = {Single-{{Particle Cryo-Electron Microscopy}}: {{Mathematical Theory}}, {{Computational Challenges}}, and {{Opportunities}}},
  author = {Bendory, Tamir and Bartesaghi, Alberto and Singer, Amit},
  year = 2020,
  month = mar,
  journal = {IEEE Signal Processing Magazine},
  volume = {37},
  number = {2},
  pages = {58--76},
  issn = {1053-5888, 1558-0792},
  doi = {10.1109/MSP.2019.2957822},
  url = {https://ieeexplore.ieee.org/document/9016106/}
}

@article{BeslMethod1992,
  title = {A Method for Registration of 3-{{D}} Shapes},
  author = {Besl, P.J. and McKay, Neil D.},
  year = 1992,
  month = feb,
  journal = {IEEE Transactions on Pattern Analysis and Machine Intelligence},
  volume = {14},
  number = {2},
  pages = {239--256},
  issn = {0162-8828, 2160-9292},
  doi = {10.1109/34.121791},
  url = {http://ieeexplore.ieee.org/document/121791/}
}

@article{BingJianRobust2011,
  title = {Robust {{Point Set Registration Using Gaussian Mixture Models}}},
  author = {{Bing Jian} and Vemuri, B C},
  year = 2011,
  month = aug,
  journal = {IEEE Transactions on Pattern Analysis and Machine Intelligence},
  volume = {33},
  number = {8},
  pages = {1633--1645},
  issn = {0162-8828},
  doi = {10.1109/TPAMI.2010.223},
  url = {http://ieeexplore.ieee.org/document/5674050/}
}

@misc{BonetSpherical2022,
  title = {Spherical {{Sliced-Wasserstein}}},
  author = {Bonet, Cl{\'e}ment and Berg, Paul and Courty, Nicolas and Septier, Fran{\c c}ois and Drumetz, Lucas and Pham, Minh-Tan},
  year = 2022,
  publisher = {arXiv},
  doi = {10.48550/ARXIV.2206.08780},
  url = {https://arxiv.org/abs/2206.08780}
}

@phdthesis{BonnotteUnidimensional2013,
  type = { {{PhD}} thesis},
  title = {Unidimensional and {{Evolution Methods}} for {{Optimal Transportation}}},
  author = {Bonnotte, Nicolas},
  year = 2013,
  month = dec,
  doi = {10.70675/2833ed85z6cd3z4115z99bazba760e7494de},
  url = {https://theses.fr/2013PA112297}
}

@article{CuturiSinkhorn2013,
  title = {Sinkhorn {{Distances}}: {{Lightspeed Computation}} of {{Optimal Transportation Distances}}},
  author = {Cuturi, Marco},
  year = 2013,
  publisher = {arXiv},
  doi = {10.48550/ARXIV.1306.0895},
  url = {https://arxiv.org/abs/1306.0895}
}

@article{GillesCryoEM2025,
  title = {Cryo-{{EM}} Heterogeneity Analysis Using Regularized Covariance Estimation and Kernel Regression},
  author = {Gilles, Marc Aur{\`e}le and Singer, Amit},
  year = 2025,
  month = mar,
  journal = {Proceedings of the National Academy of Sciences},
  volume = {122},
  number = {9},
  pages = {e2419140122},
  issn = {0027-8424, 1091-6490},
  doi = {10.1073/pnas.2419140122},
  url = {https://pnas.org/doi/10.1073/pnas.2419140122}
}

@article{HarpazThreedimensional2023,
  title = {Three-Dimensional Alignment of Density Maps in Cryo-Electron Microscopy},
  author = {Harpaz, Yael and Shkolnisky, Yoel},
  year = 2023,
  journal = {Biological Imaging},
  volume = {3},
  pages = {e8},
  issn = {2633-903X},
  doi = {10.1017/S2633903X23000089},
  url = {https://www.cambridge.org/core/product/identifier/S2633903X23000089/type/journal_article}
}

@article{HeAccurate2024,
  title = {Accurate Global and Local {{3D}} Alignment of Cryo-{{EM}} Density Maps Using Local Spatial Structural Features},
  author = {He, Bintao and Zhang, Fa and Feng, Chenjie and Yang, Jianyi and Gao, Xin and Han, Renmin},
  year = 2024,
  month = feb,
  journal = {Nature Communications},
  volume = {15},
  number = {1},
  pages = {1593},
  issn = {2041-1723},
  doi = {10.1038/s41467-024-45861-4},
  url = {https://www.nature.com/articles/s41467-024-45861-4}
}

@incollection{HutchisonUnique2010,
  title = {Unique {{Signatures}} of {{Histograms}} for {{Local Surface Description}}},
  booktitle = {Computer {{Vision}} -- {{ECCV}} 2010},
  author = {Hutchison, David and Kanade, Takeo and Kittler, Josef and Kleinberg, Jon M. and Mattern, Friedemann and Mitchell, John C. and Naor, Moni and Nierstrasz, Oscar and Pandu Rangan, C. and Steffen, Bernhard and Sudan, Madhu and Terzopoulos, Demetri and Tygar, Doug and Vardi, Moshe Y. and Weikum, Gerhard and Tombari, Federico and Salti, Samuele and Di Stefano, Luigi},
  year = 2010,
  volume = {6313},
  pages = {356--369},
  publisher = {Springer Berlin Heidelberg},
  doi = {10.1007/978-3-642-15558-1_26},
  url = {http://link.springer.com/10.1007/978-3-642-15558-1_26},
  isbn = {978-3-642-15557-4}
}

@article{KastritisCapturing2017,
  title = {Capturing Protein Communities by Structural Proteomics in a Thermophilic Eukaryote},
  author = {Kastritis, Panagiotis L and O'Reilly, Francis J and Bock, Thomas and Li, Yuanyue and Rogon, Matt Z and Buczak, Katarzyna and Romanov, Natalie and Betts, Matthew J and Bui, Khanh Huy and Hagen, Wim J and Hennrich, Marco L and Mackmull, Marie-Therese and Rappsilber, Juri and Russell, Robert B and Bork, Peer and Beck, Martin and Gavin, Anne-Claude},
  year = 2017,
  month = jul,
  journal = {Molecular Systems Biology},
  volume = {13},
  number = {7},
  pages = {MSB167412},
  issn = {1744-4292},
  doi = {10.15252/msb.20167412},
  url = {https://link.springer.com/10.15252/msb.20167412}
}

@article{KhooNonIterative2016,
  title = {Non-{{Iterative Rigid 2D}}/{{3D Point-Set Registration Using Semidefinite Programming}}},
  author = {Khoo, Yuehaw and Kapoor, Ankur},
  year = 2016,
  month = jul,
  journal = {IEEE Transactions on Image Processing},
  volume = {25},
  number = {7},
  pages = {2956--2970},
  issn = {1057-7149, 1941-0042},
  doi = {10.1109/TIP.2016.2540810},
  url = {http://ieeexplore.ieee.org/document/7430328/}
}

@misc{KolouriGeneralized2019,
  title = {Generalized {{Sliced Wasserstein Distances}}},
  author = {Kolouri, Soheil and Nadjahi, Kimia and Simsekli, Umut and Badeau, Roland and Rohde, Gustavo K.},
  year = 2019,
  publisher = {arXiv},
  doi = {10.48550/ARXIV.1902.00434},
  url = {https://arxiv.org/abs/1902.00434}
}

@inproceedings{KolouriGeneralized2022,
  title = {Generalized {{Sliced Probability Metrics}}},
  booktitle = {{{ICASSP}} 2022 - 2022 {{IEEE International Conference}} on {{Acoustics}}, {{Speech}} and {{Signal Processing}} ({{ICASSP}})},
  author = {Kolouri, Soheil and Nadjahi, Kimia and Shahrampour, Shahin and Simsekli, Umut},
  year = 2022,
  month = may,
  pages = {4513--4517},
  publisher = {IEEE},
  doi = {10.1109/ICASSP43922.2022.9746016},
  url = {https://ieeexplore.ieee.org/document/9746016/},
  isbn = {978-1-6654-0540-9}
}

@article{KolouriRadon2016,
  title = {The {{Radon Cumulative Distribution Transform}} and {{Its Application}} to {{Image Classification}}},
  author = {Kolouri, Soheil and Park, Se Rim and Rohde, Gustavo K.},
  year = 2016,
  month = feb,
  journal = {IEEE Transactions on Image Processing},
  volume = {25},
  number = {2},
  pages = {920--934},
  issn = {1057-7149, 1941-0042},
  doi = {10.1109/TIP.2015.2509419},
  url = {http://ieeexplore.ieee.org/document/7358128/}
}

@article{KovacsFast2002,
  title = {Fast Rotational Matching},
  author = {Kovacs, Julio A. and Wriggers, Willy},
  year = 2002,
  month = jul,
  journal = {Acta Crystallographica Section D Biological Crystallography},
  volume = {58},
  number = {8},
  pages = {1282--1286},
  issn = {09074449},
  doi = {10.1107/S0907444902009794},
  url = {https://journals.iucr.org/paper?S0907444902009794}
}

@misc{KruseFast2026,
  title = {Fast {{Volume Alignment}} by {{Frequency-Marched Newton Method}}},
  author = {Kruse, Fabian and Debarnot, Valentin and Kishore, Vinith and Dokmani{\'c}, Ivan},
  year = 2026,
  month = jul,
  publisher = {arXiv},
  doi = {10.48550/ARXIV.2603.15285},
  url = {https://arxiv.org/abs/2603.15285}
}

@article{KyatkinAlgorithms2000,
  title = {Algorithms for {{Fast Convolutions}} on {{Motion Groups}}},
  author = {Kyatkin, Alexander B and Chirikjian, Gregory S},
  year = 2000,
  month = sep,
  journal = {Applied and Computational Harmonic Analysis},
  volume = {9},
  number = {2},
  pages = {220--241},
  issn = {10635203},
  doi = {10.1006/acha.2000.0321},
  url = {https://linkinghub.elsevier.com/retrieve/pii/S106352030090321X}
}

@misc{LeebSliced2026,
  title = {On Sliced {{Cram\'er}} Metrics},
  author = {Leeb, William},
  year = 2026,
  month = aug,
  publisher = {arXiv},
  doi = {10.48550/ARXIV.2508.02678},
  url = {https://arxiv.org/abs/2508.02678}
}

@article{OvsjanikovFunctional2012,
  title = {Functional Maps: A Flexible Representation of Maps between Shapes},
  author = {Ovsjanikov, Maks and {Ben-Chen}, Mirela and Solomon, Justin and Butscher, Adrian and Guibas, Leonidas},
  year = 2012,
  month = aug,
  journal = {ACM Transactions on Graphics},
  volume = {31},
  number = {4},
  pages = {1--11},
  issn = {0730-0301, 1557-7368},
  doi = {10.1145/2185520.2185526},
  url = {https://dl.acm.org/doi/10.1145/2185520.2185526}
}

@article{PettersenUCSF2021,
  title = {{{{\textsc{UCSF ChimeraX}}}} : {{Structure}} Visualization for Researchers, Educators, and Developers},
  author = {Pettersen, Eric F. and Goddard, Thomas D. and Huang, Conrad C. and Meng, Elaine C. and Couch, Gregory S. and Croll, Tristan I. and Morris, John H. and Ferrin, Thomas E.},
  year = 2021,
  month = jan,
  journal = {Protein Science},
  volume = {30},
  number = {1},
  pages = {70--82},
  issn = {0961-8368, 1469-896X},
  doi = {10.1002/pro.3943},
  url = {https://onlinelibrary.wiley.com/doi/10.1002/pro.3943}
}

@article{PeyreComputational2019,
  title = {Computational {{Optimal Transport}} with {{Applications}} to {{Data Sciences}}},
  author = {Peyr{\'e}, Gabriel and Cuturi, Marco},
  year = 2019,
  month = feb,
  journal = {Foundations and Trends in Machine Learning},
  volume = {11},
  number = {5-6},
  pages = {355--607},
  issn = {1935-8237, 1935-8245},
  doi = {10.1561/2200000073},
  url = {https://www.emerald.com/ftmal/article/11/5-6/355/1332396/Computational-Optimal-Transport-with-Applications}
}

@article{QuellmalzSliced2023,
  title = {Sliced Optimal Transport on the Sphere},
  author = {Quellmalz, Michael and Beinert, Robert and Steidl, Gabriele},
  year = 2023,
  month = oct,
  journal = {Inverse Problems},
  volume = {39},
  number = {10},
  pages = {105005},
  issn = {0266-5611, 1361-6420},
  doi = {10.1088/1361-6420/acf156},
  url = {https://iopscience.iop.org/article/10.1088/1361-6420/acf156}
}

@incollection{RabinWasserstein2012,
  title = {Wasserstein {{Barycenter}} and {{Its Application}} to {{Texture Mixing}}},
  booktitle = {Scale {{Space}} and {{Variational Methods}} in {{Computer Vision}}},
  author = {Rabin, Julien and Peyr{\'e}, Gabriel and Delon, Julie and Bernot, Marc},
  year = 2012,
  volume = {6667},
  pages = {435--446},
  publisher = {Springer Berlin Heidelberg},
  doi = {10.1007/978-3-642-24785-9_37},
  url = {http://link.springer.com/10.1007/978-3-642-24785-9_37},
  isbn = {978-3-642-24784-2}
}

@article{RiahiAlignOT2023,
  title = {{{AlignOT}}: {{An Optimal Transport Based Algorithm}} for {{Fast 3D Alignment With Applications}} to {{Cryogenic Electron Microscopy Density Maps}}},
  author = {Riahi, Aryan Tajmir and Woollard, Geoffrey and Poitevin, Fr{\'e}d{\'e}ric and Condon, Anne and Duc, Khanh Dao},
  year = 2023,
  month = nov,
  journal = {IEEE/ACM Transactions on Computational Biology and Bioinformatics},
  volume = {20},
  number = {6},
  pages = {3842--3850},
  issn = {1545-5963, 1557-9964, 2374-0043},
  doi = {10.1109/TCBB.2023.3327633},
  url = {https://ieeexplore.ieee.org/document/10298813/}
}

@article{ShiFast2026,
  title = {Fast {{Rigid Alignment}} of {{Heterogeneous Images}} in {{Sliced Wasserstein Distance}}},
  author = {Shi, Yunpeng and Singer, Amit and Verbeke, Eric J.},
  year = 2026,
  month = sep,
  journal = {SIAM Journal on Imaging Sciences},
  volume = {19},
  number = {3},
  pages = {1469--1515},
  issn = {1936-4954},
  doi = {10.1137/25M1742916},
  url = {https://epubs.siam.org/doi/10.1137/25M1742916}
}

@article{SingerAlignment2024,
  title = {Alignment of Density Maps in {{Wasserstein}} Distance},
  author = {Singer, Amit and Yang, Ruiyi},
  year = 2024,
  journal = {Biological Imaging},
  volume = {4},
  pages = {e5},
  issn = {2633-903X},
  doi = {10.1017/S2633903X24000059},
  url = {https://www.cambridge.org/core/product/identifier/S2633903X24000059/type/journal_article}
}

@article{SingerComputational2020,
  title = {Computational {{Methods}} for {{Single-Particle Electron Cryomicroscopy}}},
  author = {Singer, Amit and Sigworth, Fred J.},
  year = 2020,
  month = jul,
  journal = {Annual Review of Biomedical Data Science},
  volume = {3},
  number = {1},
  pages = {163--190},
  issn = {2574-3414, 2574-3414},
  doi = {10.1146/annurev-biodatasci-021020-093826},
  url = {https://www.annualreviews.org/doi/10.1146/annurev-biodatasci-021020-093826}
}

@article{SolomonConvolutional2015,
  title = {Convolutional Wasserstein Distances: Efficient Optimal Transportation on Geometric Domains},
  author = {Solomon, Justin and De Goes, Fernando and Peyr{\'e}, Gabriel and Cuturi, Marco and Butscher, Adrian and Nguyen, Andy and Du, Tao and Guibas, Leonidas},
  year = 2015,
  month = jul,
  journal = {ACM Transactions on Graphics},
  volume = {34},
  number = {4},
  pages = {1--11},
  issn = {0730-0301, 1557-7368},
  doi = {10.1145/2766963},
  url = {https://dl.acm.org/doi/10.1145/2766963}
}

@article{TajmirRiahiAlignment2025,
  title = {Alignment of {{Partially Overlapping Cryo-EM Maps Using Unbalanced Gromov-Wasserstein Divergence}}},
  author = {Tajmir Riahi, Aryan and Zhang, Chenwei and Condon, Anne and Chen, James and Dao Duc, Khanh},
  year = 2025,
  month = apr,
  journal = {PRX Life},
  volume = {3},
  number = {2},
  pages = {023003},
  issn = {2835-8279},
  doi = {10.1103/PRXLife.3.023003},
  url = {https://link.aps.org/doi/10.1103/PRXLife.3.023003}
}

@misc{TranStereographic2024,
  title = {Stereographic {{Spherical Sliced Wasserstein Distances}}},
  author = {Tran, Huy and Bai, Yikun and Kothapalli, Abihith and Shahbazi, Ashkan and Liu, Xinran and Martin, Rocio Diaz and Kolouri, Soheil},
  year = 2024,
  publisher = {arXiv},
  doi = {10.48550/ARXIV.2402.02345},
  url = {https://arxiv.org/abs/2402.02345}
}

@article{VerbekeClassification2018,
  title = {Classification of {{Single Particles}} from {{Human Cell Extract Reveals Distinct Structures}}},
  author = {Verbeke, Eric J. and Mallam, Anna L. and Drew, Kevin and Marcotte, Edward M. and Taylor, David W.},
  year = 2018,
  month = jul,
  journal = {Cell Reports},
  volume = {24},
  number = {1},
  pages = {259-268.e3},
  issn = {22111247},
  doi = {10.1016/j.celrep.2018.06.022},
  url = {https://linkinghub.elsevier.com/retrieve/pii/S2211124718309173}
}

@book{VillaniOptimal2009,
  title = {Optimal {{Transport}}: {{Old}} and {{New}}},
  author = {Villani, C{\'e}dric},
  year = 2009,
  volume = {338},
  publisher = {Springer Berlin Heidelberg},
  doi = {10.1007/978-3-540-71050-9},
  url = {http://link.springer.com/10.1007/978-3-540-71050-9},
  isbn = {978-3-540-71049-3}
}

@misc{WoollardImproving2025,
  title = {Improving {{Cryo-EM Optimization Robustness}} with an {{Optimal Transport Loss Function}} for {{Noisy Images}}},
  author = {Woollard, Geoffrey and Herreros, David and Li, Minhuan and Cossio, Pilar and Duc, Khanh Dao},
  year = 2025,
  month = dec,
  publisher = {bioRxiv},
  doi = {10.64898/2025.12.23.696001},
  url = {http://biorxiv.org/lookup/doi/10.64898/2025.12.23.696001}
}

@inproceedings{ZhirongWu3D2015,
  title = {{{3D ShapeNets}}: {{A}} Deep Representation for Volumetric Shapes},
  booktitle = {2015 {{IEEE Conference}} on {{Computer Vision}} and {{Pattern Recognition}} ({{CVPR}})},
  author = {{Zhirong Wu} and Song, Shuran and Khosla, Aditya and {Fisher Yu} and {Linguang Zhang} and {Xiaoou Tang} and Xiao, Jianxiong},
  year = 2015,
  month = jun,
  pages = {1912--1920},
  publisher = {IEEE},
  doi = {10.1109/CVPR.2015.7298801},
  url = {https://ieeexplore.ieee.org/document/7298801/},
  isbn = {978-1-4673-6964-0}
}

@article{ZhongCryoDRGN2021,
  title = {{{CryoDRGN}}: Reconstruction of Heterogeneous Cryo-{{EM}} Structures Using Neural Networks},
  author = {Zhong, Ellen D. and Bepler, Tristan and Berger, Bonnie and Davis, Joseph H.},
  year = 2021,
  month = feb,
  journal = {Nature Methods},
  volume = {18},
  number = {2},
  pages = {176--185},
  issn = {1548-7091, 1548-7105},
  doi = {10.1038/s41592-020-01049-4},
  url = {https://www.nature.com/articles/s41592-020-01049-4}
}

@article{ThewwPDBConsortium2024EMDB,
  title = {{{EMDB}}---the {{Electron Microscopy Data Bank}}},
  author = {{The wwPDB Consortium} and Turner, Jack and Abbott, Sanja and Fonseca, Neli and Pye, Ryan and Carrijo, Lucas and Duraisamy, Amudha Kumari and Salih, Osman and Wang, Zhe and Kleywegt, Gerard J and Morris, Kyle L and Patwardhan, Ardan and Burley, Stephen K and Crichlow, Gregg and Feng, Zukang and Flatt, Justin W and Ghosh, Sutapa and Hudson, Brian P and Lawson, Catherine L and Liang, Yuhe and Peisach, Ezra and Persikova, Irina and Sekharan, Monica and Shao, Chenghua and Young, Jasmine and Velankar, Sameer and Armstrong, David and Bage, Marcus and Bueno, Wesley Morellato and Evans, Genevieve and Gaborova, Romana and Ganguly, Sudakshina and Gupta, Deepti and Harrus, Deborah and Tanweer, Ahsan and Bansal, Manju and Rangannan, Vetriselvi and Kurisu, Genji and Cho, Hasumi and Ikegawa, Yasuyo and Kengaku, Yumiko and Kim, Ju Yaen and Niwa, Satomi and Sato, Junko and Takuwa, Ayako and Yu, Jian and Hoch, Jeffrey C and Baskaran, Kumaran and Xu, Wenqing and Zhang, Weizhe and Ma, Xiaodan},
  year = 2024,
  month = jan,
  journal = {Nucleic Acids Research},
  volume = {52},
  number = {D1},
  pages = {D456-D465},
  issn = {0305-1048, 1362-4962},
  doi = {10.1093/nar/gkad1019},
  url = {https://academic.oup.com/nar/article/52/D1/D456/7442543}
}

@misc{GarrettWright2026ComputationalCryoEM,
  title = {{{ComputationalCryoEM}}/{{ASPIRE-Python}}: V0.14.3},
  author = {Garrett Wright and Joakim And{\'e}n and Vineet Bansal and Junchao Xia and Chris Langfield and Josh Carmichael and Kris Sowattanangkul and Robbie Brook and Yunpeng Shi and Ayelet Heimowitz and Gabi Pragier and Itay Sason and Amit Moscovich and Yoel Shkolnisky and Amit Singer},
  year = 2026,
  month = apr,
  doi = {10.5281/ZENODO.5657281},
  url = {https://zenodo.org/doi/10.5281/zenodo.5657281},
  howpublished = {Zenodo}
}

@book{Kak2001Principles,
  title = {Principles of {{Computerized Tomographic Imaging}}},
  author = {Kak, Avinash C. and Slaney, Malcolm},
  year = 2001,
  month = jan,
  publisher = {{Society for Industrial and Applied Mathematics}},
  doi = {10.1137/1.9780898719277},
  url = {https://epubs.siam.org/doi/book/10.1137/1.9780898719277}
}

\section*{Acknowledgments}
A.S. was supported in part by AFOSR FA9550-23-1-0249, DARPA HR0011-25-3-E002, NSF DMS 2510039,  Biohub FP00007575, and the Simons Foundation Math+X Investigator Award; Y.S. was supported by NSF DMS 2514152.

\end{document}